\documentclass[english,10pt]{amsart}
\usepackage[english]{babel}

\usepackage[matrix,arrow]{xy}
\xyoption{all}
\usepackage{amscd,amssymb,amsfonts,amsmath}

\usepackage{graphics}
\usepackage{epsfig}
\usepackage{mathrsfs}

\newcommand\mypagesizel{
\textwidth= 6.5in
\textheight=9in
\voffset-.55in
\hoffset -0.75in
\marginparwidth=56pt
}

\newcommand{\codim}{\textup{codim}}

\newcommand{\N}{\textup{N}}

\renewcommand{\phi}{\varphi}
\newcommand{\into}{\hookrightarrow}
\newcommand{\map}{\dashrightarrow}

\renewcommand{\le}{\leqslant}
\renewcommand{\ge}{\geqslant}

\newcommand{\bQ}{\mathbb{Q}}

\newcommand{\cQ}{\mathcal{Q}}

\newcommand{\sD}{\mathscr{D}}
\newcommand{\sE}{\mathscr{E}}
\newcommand{\sF}{\mathscr{F}}
\newcommand{\sG}{\mathscr{G}}
\newcommand{\sH}{\mathscr{H}}

\newcommand{\sK}{\mathscr{K}}
\newcommand{\sL}{\mathscr{L}}

\newcommand{\sN}{\mathscr{N}}
\newcommand{\sO}{\mathscr{O}}

\newcommand{\sS}{\mathscr{S}}
\newcommand{\sT}{\mathscr{T}}

\mypagesizel

\newtheorem{thm}{Theorem}[section]

\newtheorem{question}[thm]{Question}
\newtheorem{lemma}[thm]{Lemma}
\newtheorem{cor}[thm]{Corollary}

\newtheorem{prop}[thm]{Proposition}

\newtheorem*{thm*}{Theorem}
\theoremstyle{definition}
\newtheorem{defn}[thm]{Definition}

\newtheorem{say}[thm]{}
\newtheorem{exmp}[thm]{Example}

\newtheorem{notation}[thm]{Notation}

\newtheorem{defn-thm}[thm]{Definition-Theorem} 
\newtheorem{defn-lemma}[thm]{Definition-Lemma}

\newtheorem{rem}[thm]{Remark}

\theoremstyle{remark}

\newtheorem*{not-and-def}{Notation and definitions}

\numberwithin{equation}{section}

\def\factor#1.#2.{\left. \raise 2pt\hbox{$#1$} \right/\hskip -2pt\raise -2pt\hbox{$#2$}}

\begin{document}

\title[On foliations with nef anti-canonical bundle]{On foliations with nef anti-canonical bundle}

\author{St\'ephane \textsc{Druel}}

\address{St\'ephane Druel: Institut Fourier, UMR 5582 du
  CNRS, Universit\'e Grenoble 1, BP 74, 38402 Saint Martin
  d'H\`eres, France} 

\email{druel@ujf-grenoble.fr}


\begin{abstract}
In this paper we prove that the anti-canonical bundle of a holomorphic foliation $\sF$
on a complex projective manifold cannot be nef and big if either $\sF$ is regular, or $\sF$ has a compact leaf. Then we address codimension one regular foliations 
whose anti-canonical bundle is nef with maximal Kodaira dimension. 
\end{abstract}

\maketitle

{\small\tableofcontents}

\section{Introduction}

In the last few decades, much progress has been made in the classification of complex projective varieties.
The general viewpoint is that complex projective manifolds $X$ should be classified according to the behavior of their canonical class $K_X$.
Similar ideas can be applied in the context of \emph{foliations} on complex projective manifolds.
If $\sF\subset T_X$ is a foliation on a  complex projective manifold, we define its canonical class to be 
$K_{\sF}=-c_1(\sF)$. In analogy with the case of projective manifolds, one expects the numerical properties of $K_{\sF}$
to reflect geometric aspects of $\sF$ (see for instance \cite{druel04}, \cite{loray_pereira_touzet}, \cite{fano_fols}, \cite{codim_1_del_pezzo_fols}, \cite{fano_fols_2}).

In \cite{popa_schnell}, Popa and Schnell proved that the canonical bundle of a codimension one regular foliation 
with trivial normal bundle cannot be big.
In this paper we propose to investigate regular foliations on complex projective manifolds with $-K_\sF$ nef. Codimension one regular foliations with trivial canonical bundle were classified by Touzet in \cite{touzet}. More recently, Pereira and Touzet have investigated regular foliations $\sF$ of arbitrary 
rank on complex projective manifolds with $c_1(\sF)=0$ and $c_2(\sF)=0$.

In  \cite[Theorem 2]{miyaoka93}, Miyaoka proved that the anticanonical bundle of a smooth 
projective morphism $f:X\to C$ onto a smooth proper curve cannot be ample. 
In \cite[Proposition 1]{zhang} (see also \cite[Theorem 3.12]{debarre}),
this result was generalized by dropping the smoothness assumption on $f$.
In this note, we give a further generalization of this result to foliations (see also Theorem \ref{thm:main2}).

\begin{thm}\label{thm:main}
Let $X$ be a complex projective manifold, and let $\sF\subset T_X$ be a 
codimension $q$ foliation
with $0 < q < \dim X$. 
Suppose that either $\sF$ is regular, or that $\sF$ has a compact leaf.
Then $-K_\sF$ is not nef and big.
\end{thm}

In \cite[Theorem 1.5]{codim_1_del_pezzo_fols}, the authors proved that the anti-canonical bundle of a codimension one foliation on a complex projective manifold whose singular set has codimension $\ge 3$ cannot be nef and big.

The following example show that the statement of Theorem \ref{thm:main} becomes wrong if one relaxes the assumption on $K_\sF$.

\begin{exmp}Fix an integer $m$ such that $m\ge 2$.
Let $A$ be an abelian variety with $\dim A \ge 1$, and let $\sL$ be an ample line bundle on $A$. Set 
$X:=\mathbb{P}_A(\sL\oplus\sL^{\otimes m})$
with natural morphism $\phi\colon X \to A$, and tautological line bundle $\sO_X(1)$. 
We have 
$$\sO_X(-K_{X/A})\cong \big(\sO_X(1)\otimes
\phi^*\sL^{\otimes -1}\big)\otimes\big(\sO_X(1)\otimes\phi^*\sL^{\otimes -m}\big).$$
Note that $h^0\big(X,\sO_X(1)\otimes\phi^*\sL^{\otimes -m}\big) = 1$, and that for any positive integer $k$, we have 
$$
\begin{array}{ccll}
h^0\big(X,\sO_X(k)\otimes\phi^*\sL^{\otimes -k}\big) & = & 
h^0\Big(A,\textup{S}^k\big(\sO_A\oplus\sL^{\otimes m-1}\big)\Big)
& \\
& = &  \sum_{0\le i\le k} h^0\big(A,\sL^{\otimes i(m-1)}\big)
& \text{ by the projection formula,}\\
& =  & \frac{1}{(n-1)!}\sum_{0\le i\le k}\big(i(m-1)\big)^{n-1}\sL^{n-1}
& \text{ by the G-R-R theorem,}\\
& \approx  & \frac{(m-1)^{n-1}\sL^{n-1}}{n(n-1)!}k^n  & \text{ by Faulhaber's formula.}\\
\end{array}
$$
Thus $\sO_X(1)\otimes\phi^*\sL^{\otimes -1}$ is big, and so is $-K_{X/A}$.
Let $\sG$ be a linear foliation on $A$ of codimension $0<q\le \dim A$, and let
$\sF$ be the pull-back of $\sG$ via $\varphi$ (see \ref{pullback_foliations}). Then 
$\sF$ is a regular foliation with $-K_\sF\sim -K_{X/A}$ big.
\end{exmp}

The next example shows that Theorem \ref{thm:main} is wrong if one drops the integrability assumption on
$\sF$.

\begin{exmp}
The null correlation bundle $\sN$ on $\mathbb{P}^{2n+1}$ (see \cite[4.2]{OSS}) yields a contact distribution
$\sD=\sN\otimes\sO_{\mathbb{P}^{2n+1}}(1)\subset T_{\mathbb{P}^{2n+1}}$
on $\mathbb{P}^{2n+1}$ corresponding to a twisted $1$-form
$\theta\in H^0\big(\mathbb{P}^{2n+1},\Omega_{\mathbb{P}^{2n+1}}^1\otimes\sO_{\mathbb{P}^{2n+1}}(2)\big)$.

Recall that a contact structure on a complex manifold $X$ is a corank 1 subbundle 
$\sD\subset T_X$ such that the bilinear form on $\sD$ with values in the quotient line bundle 
$\sL=T_X/\sD$
deduced from the Lie bracket on $T_X$
is everywhere non-degenerate. This implies
that the dimension of $X$ is odd, say $\dim X = 2n+1$, that the canonical
bundle $\sO_X(K_X)$
is isomorphic to $\sL^{\otimes n-1}$, and that $\det(\sD)\cong \sL^{\otimes n}$.
Alternatively, the contact structure can be described by the twisted $1$-form 
$\theta\in H^0(X,\Omega_X^1\otimes\sL)$ corresponding to the natural projection
$T_X \twoheadrightarrow \sL$. 
\end{exmp}

An immediate consequence of Theorem \ref{thm:main} is the following upper bound on the Kodaira dimension of the anti-canonical bundle of a holomorphic codimension 1 foliation.

\begin{cor}\label{cor:kodaira_codim1}
Let $X$ be a complex projective manifold, and let $\sF\subset T_X$ be a 
codimension $1$ foliation. 
Suppose that either $\sF$ is regular, or that $\sF$ has a compact leaf.
Suppose 
furthermore
that $-K_\sF$ is nef. Then $\kappa(X,-K_\sF) \le \dim X-1$.
\end{cor}

Next, we investigate regular codimension one foliations with $-K_\sF$ is nef and 
$\kappa(X,-K_\sF) = \dim X-1$. Note that $-K_\sF$ is then nef and abundant.

\begin{thm}\label{thm:main3}
Let $X$ be a complex projective manifold, 
and let $\sF\subset T_X$ be a
codimension $q$ foliation
with $0 < q < \dim X$. 
Suppose that either $\sF$ is regular, or that $\sF$ has a compact leaf.
Suppose 
furthermore
that $-K_\sF $ is nef and abundant.
Then $\kappa(X,-K_\sF) \le \dim X-q$, and equality holds only if 
$\sF$ is algebraically integrable.
\end{thm}

The following example shows to what extend Theorem \ref{thm:main3} is optimal.

\begin{exmp}
Let $T$ be a positive-dimensional
projective manifold with $-K_T$ nef, and let $A$ be a positive-dimensional abelian variety. Consider a linear foliation $\sG$ on $A$, and let 
$\sF$ be the pull-back of $\sG$ on $T\times A$.
Then $-K_\sF\sim -K_{T \times A /A}$ is nef and 
$\nu(-K_\sF)=\kappa(T\times A,-K_\sF)=\kappa(T,-K_T)=\nu(-K_T)$. Moreover, if $\sG$ is general enough, then $\sF$ has no algebraic leaf.
\end{exmp}

\begin{thm}\label{thm:main4}
Let $X$ be a complex projective manifold with $h^1(X,\sO_X)=0$, and let $\sF\subset T_X$ be a 
regular codimension $1$ foliation. 
Suppose 
that $-K_\sF$ is nef and $\kappa(X,-K_\sF) = \dim X-1$. Then $X\cong \mathbb{P}^1\times F$, and $\sF$ is induced by the natural morphism $X \cong \mathbb{P}^1\times F \to \mathbb{P}^1$.
\end{thm}

The proof of the main results rely on the following observation (see also 
Proposition \ref{proposition:regular_versus_algebraically_integrable_foliation}).

\begin{prop}\label{proposition:regular_versus_algebraically_integrable_foliation_intro}
Let $X$ be a complex projective manifold, and let $\sF\subset T_X$ be a foliation. 
Suppose that 
$-K_\sF$ is nef.
Suppose furthermore that either $\sF$ is regular, or that $\sF$ has a compact leaf.
There exist a foliation $\sH$ on $X$ induced by
an almost proper map $\phi\colon X \dashrightarrow Y$, and a foliation $\sG$ on $Y$ such that 
\begin{enumerate}
\item there is no positive-dimensional algebraic subvariety passing through a general point of $Y$ that is tangent to $\sG$,
\item $\sF=\phi^{-1}\sG$, and
\item $K_\sH\equiv K_\sF$.
\end{enumerate}
\end{prop}

The paper is organized as follows. In section 2, we review basic definitions and results about holomorphic foliations. In section 3, we extend a number of known results on slope-semistable sheaves from the classical case to the setting where polarisations are given by movable complete intersection curve classes.
As application, we obtain a generalization of Metha-Ramananthan's theorem.
In sections 4 and 5, we study the anti-canonical divisor of algebraically integrable foliations, and provide applications to the study of singularities of those foliations with $-K_\sF$ nef.
Section 6 is devoted to the proof of Proposition \ref{proposition:regular_versus_algebraically_integrable_foliation_intro}.
Section 7 is devoted to the proofs of Theorem \ref{thm:main}
and Corollary \ref{cor:kodaira_codim1}. In section 8 we prove Theorems \ref{thm:main3}
and \ref{thm:main4}.

\medskip

We work over the field ${\mathbb C}$ of complex numbers.

\medskip

\noindent {\bf Acknowledgements.} We would like to thank Beno\^{i}t \textsc{Claudon} for helpful discussions.

\section{Recollection: Foliations}

In this section we recall the basic facts concerning foliations.

\subsection{Foliations}

\begin{defn}
A \emph{foliation} on  a normal  variety $X$ is a coherent subsheaf $\sF\subset T_X$ such that
\begin{itemize}
	\item $\sF$ is closed under the Lie bracket, and
	\item $\sF$ is saturated in $T_X$. In other words, the quotient $T_X/\sF$ is torsion-free.
\end{itemize}

The \emph{rank} $r$ of $\sF$ is the generic rank of $\sF$.
The \emph{codimension} of $\sF$ is defined as $q:=\dim X-r$. 

Let $X^\circ \subset X_{\textup{ns}}$ be the maximal open set where $\sF_{|X_{\textup{ns}}}$ is a subbundle of $T_{X_{\textup{ns}}}$. 
A \emph{leaf} of $\sF$ is a connected, locally closed holomorphic submanifold $L \subset X^\circ$ such that
$T_L=\sF_{|L}$. A leaf is called \emph{algebraic} if it is open in its Zariski closure.
\end{defn}

\begin{defn}
Let $X$ be a complex manifold, and let $\sF$ be a foliation on $X$. 
We say that
$\sF$ is \emph{regular} if $\sF$ is a subbundle of $T_X$.
\end{defn}

Next, we define the \emph{algebraic} and \emph{transcendental} parts of
a holomorphic foliation (see \cite[Definition 2]{codim_1_del_pezzo_fols}).

\begin{defn}
Let $\sF$ be a holomorphic foliation on a normal variety $X$.
There exist  a normal variety $Y$, unique up to birational equivalence,  a dominant rational map with connected fibers $\varphi:X\map Y$,
and a holomorphic foliation $\sG$ on $Y$ such that the following holds (see \cite[Section 2.4]{loray_pereira_touzet}).
\begin{enumerate}
	\item $\sG$ is purely transcendental, i.e., there is no positive-dimensional algebraic subvariety through a general point of $Y$ that is tangent to $\sG$; and
	\item $\sF$ is the pullback of $\sG$ via $\varphi$ (see \ref{pullback_foliations} for this notion).
\end{enumerate}
The foliation on $X$ induced by $\varphi$ is called the  \emph{algebraic part} of $\sF$.
\end{defn}

\begin{say}[Foliations defined by $q$-forms] \label{q-forms}

Let $\sF$ be a codimension $q$ foliation on an $n$-dimensional complex manifold $X$. Suppose that $q \ge 1$.
The \emph{normal sheaf} of $\sF$ is $\sN:=(T_X/\sF)^{**}$.
The $q$-th wedge product of the inclusion
$\sN^*\into \Omega^1_X$ gives rise to a nonzero global section 
 $\omega\in H^0\big(X,\Omega^{q}_X\otimes \det(\sN)\big)$
whose zero locus has codimension at least $2$ in $X$. 
Such $\omega$ is \emph{locally decomposable} and \emph{integrable}.
To say that $\omega$ is locally decomposable means that, 
in a neighborhood of a general point of $X$, $\omega$ decomposes as the wedge product of $q$ local $1$-forms 
$\omega=\omega_1\wedge\cdots\wedge\omega_q$.
To say that it is integrable means that for this local decomposition one has 
$d\omega_i\wedge \omega=0$ for every $i\in\{1,\ldots,q\}$. 
The integrability condition for $\omega$ is equivalent to the condition that $\sF$ 
is closed under the Lie bracket.

Conversely, let $\sL$ be a line bundle on $X$, $q\ge 1$, and 
$\omega\in H^0(X,\Omega^{q}_X\otimes \sL)$ a global section
whose zero locus has codimension at least $2$ in $X$.
Suppose that $\omega$  is locally decomposable and integrable.
Then  one defines 
a foliation of rank $r=n-q$ on $X$ as the kernel
of the morphism $T_X \to \Omega^{q-1}_X \otimes \sL$ given by the contraction with $\omega$. 
These constructions are inverse of each other. 
 \end{say}

We will use the following notation.

\begin{notation}
Let $\phi \colon X \to Y$ be a dominant morphism of normal varieties. 
Assume either that $K_Y$ is $\mathbb{Q}$-Cartier or that $\phi$ is equidimensional.
Write
$K_{X/Y}:=K_X-\phi^*K_Y$. We refer to it as the \emph{relative canonical divisor of $X$ over $Y$}.
\end{notation}

\begin{rem}
Let $\phi\colon X \to Y$ be an equidimensional morphism of normal varieties, and let $D$ be a Weil $\mathbb{Q}$-divisor on $Y$. The pull-back $\phi^*D$ of $D$ is defined as follows. We define 
$\phi^*D$ to be the unique $\mathbb{Q}$-divisor on $X$ whose restriction to 
$\phi^{-1}(Y_{\textup{ns}})$ is $(\phi_{|\textup{ns}})^*D_{|\textup{ns}}$. This construction agrees with the usual pull-back if $D$ itself is $\mathbb{Q}$-Cartier.
\end{rem}

\begin{notation}
Let $\phi \colon X \to Y$ be a dominant morphism of normal varieties. 
Let $Y^\circ\subset Y$ be a dense open subset with $\codim\, Y\setminus Y^\circ \ge 2$ such that 
$\phi$ restricts to an equidimensional morphism
$\varphi^\circ:X^\circ\to Y^\circ$,
where $X^\circ:=\phi^{-1}(Y^\circ)$.
Set $$R(\phi^\circ)=\sum_{D^\circ} \Big((\phi^\circ)^*D^\circ-{\big((\phi^\circ)^*D^\circ\big)}_{\textup{red}}\Big)$$
where $D^\circ$ runs through all prime divisors on $Y^\circ$, and let $R(\phi)$ denotes the Zariski closure of $R(\phi^\circ)$ in $X$. We refer to it
as the \emph{ramification divisor of $\phi$}.
\end{notation}

\begin{defn}
Let $\sF$ be a foliation on a normal projective variety $X$.
The \textit{canonical class} $K_{\sF}$ of $\sF$ is any Weil divisor on $X$ such that  $\sO_X(-K_{\sF})\cong \det(\sF)$. 

We say that $\sF$ is \emph{$\mathbb{Q}$-Gorenstein} if $K_\sF$ is a $\mathbb{Q}$-Cartier divisor.

\end{defn}

\begin{say}[Foliations described as pull-backs] \label{pullback_foliations}

Let $X$ and $Y$ be normal varieties, and let $\varphi\colon X\map Y$ be a dominant rational map that restricts to a morphism $\varphi^\circ\colon X^\circ\to Y^\circ$,
where $X^\circ\subset X$ and  $Y^\circ\subset Y$ are smooth open subsets.

Let $\sG$ be a codimension $q$ foliation on $Y$ 
with $q \ge 1$. Suppose that the restriction $\sG^\circ$ of $\sG$ to $Y^\circ$ is
defined by a twisted $q$-form
$\omega_{Y^\circ}\in H^0\big(Y^\circ,\Omega^{q}_{Y^\circ}\otimes \det(\sN_{\sG^\circ})\big)$.
Then $\omega_{Y^\circ}$ induces a nonzero twisted $q$-form 
$\omega_{X^\circ}\in 
H^0\Big(X^\circ,\Omega^{q}_{X^\circ}\otimes (\varphi^\circ)^*\big(\det(\sN_\sG)_{|Y^\circ}\big)\Big)$, which in turn defines a codimension $q$ foliation $\sF^\circ$ on $X^\circ$. We say that
the saturation $\sF$ of $\sF^\circ$ in $T_X$
\emph{is the pull-back of $\sG$ via $\varphi$},
and write $\sF=\varphi^{-1}\sG$.

Suppose that $X^\circ$ is such that $\varphi^\circ$ is an equidimensional morphism. Let 
$(B_i)_{i\in I}$
be the (possibly empty) set of hypersurfaces in $Y^\circ$ contained in the set of critical values of $\varphi^\circ$ and invariant by $\sG$. A straightforward computation shows that 
\begin{equation}\label{pullback_fol}
\det\big(\sN_{\sF^\circ}\big) \ \cong \ (\varphi^\circ)^*\det\big(\sN_{\sG_{|Y^\circ}}\big)
\otimes\sO_{X^\circ}\Big(\sum_{i\in I}\big((\varphi^\circ)^*B_i\big)_{\textup{red}}-(\varphi^\circ)^*B_i\Big).
\end{equation}
In particular, if
$\sF$ is induced by $\phi$, then \eqref{pullback_fol} gives
\begin{equation}\label{morphism_fol}
K_{\sF^\circ}=K_{X^\circ/Y^\circ}-R(\phi^\circ),
\end{equation}
where $R(\phi^\circ)$ denotes the ramification divisor of $\phi^\circ$.

Conversely, let $\sF$ be a foliation on $X$, and suppose that the general fiber of $\varphi$ is tangent to $\sF$. This means that, for a general point $x$ on a general fiber $F$ of $\varphi$,
the linear subspace $\sF_x\subset T_xX$ determined by the inclusion $\sF\subset T_X$ 
contains $T_xF$. Suppose moreover that $\varphi^\circ$ is smooth 
with connected fibers. Then, by \cite[Lemma 6.7]{fano_fols}, there is a  holomorphic foliation $\sG$ on $Y$
such that $\sF=\varphi^{-1}\sG$. 
\end{say}

\subsection{Algebraically integrable foliations}

\begin{defn}
Let $X$ be normal variety. A foliation $\sF$ on $X$ is said to be
\emph{algebraically integrable} if  the leaf of $\sF$ through a general point of $X$ is an algebraic variety. 
\end{defn}

\begin{defn}\label{log_leaf} 
Let $\sF$ be an algebraically integrable 
$\mathbb{Q}$-Gorenstein
foliation on a normal projective variety $X$.
Let $i:\ F\to  X$ be the normalization of the closure of a general leaf of $\sF$. 
There is an effective $\mathbb{Q}$-divisor
$\Delta$ on $ F$ such that
$K_{ F}  +  \Delta \sim i^*K_{\sF} $ (\cite[Definition 3.6]{fano_fols_2}). 
The pair $(F,\Delta)$ is called a \emph{general log leaf} of $\sF$.  
\end{defn}
The case when $(F,\Delta)$ is \emph{log canonical} is specially interesting (see \cite{fano_fols} and \cite{fano_fols_2}).
We refer to  \cite[section 2.3]{kollar_mori} for the definition of klt and log canonical pairs.
Here we only remark that if $F$ is smooth and $\Delta$ is a reduced simple normal crossing divisor, then 
$(F,\Delta)$ is log canonical.

The same argument used in the proof of 
\cite[Proposition 3.13]{fano_fols_2} shows that the following holds. We leave the details to the reader. In particular, 
Proposition \ref{proposition:lc_center} below says that an algebraically integrable $\mathbb{Q}$-Gorenstein foliation 
with mild singularities and big anti-canonical divisor has a very special property: there is a common point contained in the closure of a general leaf.

\begin{prop}\label{proposition:lc_center}
Let $X$ be a normal projective variety, let $\sF$ be a $\mathbb{Q}$-Gorenstein algebraically integrable 
foliation on $X$, and let $(F,\Delta)$ be its general log leaf. 
Suppose that $-K_\sF=A+E$ where 
$A$ is a $\mathbb{Q}$-ample $\mathbb{Q}$-divisor and
$E$ is an effective $\mathbb{Q}$-divisor. Suppose furthermore that
$(F,\Delta+E_{|F})$ is log canonical.
Then there is a closed irreducible subset $T\subset X$
satisfying the following property. For a general log leaf $(F,\Delta)$, there exists a log canonical center S of $(F,\Delta+E_{|F})$
whose image in $X$ is $T$.
\end{prop}

\begin{say}[The family of log leaves] \label{family_leaves} 
Let $X$ be normal projective variety, and let $\sF$ be an algebraically integrable foliation on $X$.
We describe the \emph{family of leaves} of $\sF$
(see \cite[Remark 3.12]{codim_1_del_pezzo_fols}).
There is a unique normal projective variety $Y$ contained in the normalization 
of the Chow variety of $X$ 
whose general point parametrizes the closure of a general leaf of $\sF$
(viewed as a reduced and irreducible cycle in $X$).
Let $Z \to Y\times X$ denotes the normalization of the universal cycle.
It comes with morphisms:

\centerline{
\xymatrix{
Z \ar[r]^{\nu}\ar[d]_{\psi} & X, \\
 Y &
}
}
\noindent where $\nu\colon Z\to X$ is birational and, for a general point $y\in Y$, 
$\nu\big(\psi^{-1}(y)\big) \subset X$ is the closure of a leaf of $\sF$.
The variety $Y$ is called the \emph{family of leaves} of $\sF$.

Suppose moreover that $\sF$ is $\bQ$-Gorenstein. Denote by $\sF_Z$ the foliation induced by $\sF$ (or $\psi$) on $Z$.
There is a canonically defined effective Weil $\bQ$-divisor $B$ on $Z$ such that 
\begin{equation}\label{eq:universal_canonical_bundle_formula}
K_{\sF_Z}+B=K_{Z/Y}-R(\psi)+B \sim_\mathbb{Q} \nu^* K_\sF,
\end{equation}
where $R(\psi)$ denotes the ramification divisor of $\psi$. Note that the equality
$K_{\sF_Z}=K_{Z/Y}-R(\psi)$ follows from from \eqref{morphism_fol}.
Suppose that $y\in Y$ is a general point, and
set $Z_y := \psi^{-1}(y)$ and $\Delta_y:=B_{|U_y}$.
Then $(Z_y, \Delta_y)$ coincides with the general log leaf $(F,\Delta)$ defined above.
\end{say}

\begin{rem}In the setup of \ref{family_leaves}, 
we claim that $B$ is $\nu$-exceptional.
This is an immediate consequence of the equality $\nu_*K_{\sF_Z}=K_\sF$.
\end{rem}

\begin{rem}
In the setup of \ref{family_leaves}, suppose furthermore that 
$K_\sF$ is a Cartier divisor. Then $B$ is a Weil divisor, and 
\eqref{eq:universal_canonical_bundle_formula} reads 
$$K_{\sF_Z}+B=K_{Z/Y}-R(\psi)+B \sim \nu^* K_\sF.$$
\end{rem}

We end this subsection with a useful criterion of algebraic integrability for foliations.

\begin{thm}[{\cite[Theorem 0.1]{bogomolov_mcquillan01}, \cite[Theorem 3.5]{bost}}] \label{thm:BM}
Let $X$ be a normal complex projective variety, and let $\sF$ be a  foliation on $X$.
Let $C \subset X$ be a complete curve disjoint from the singular loci of $X$ and $\sF$.
Suppose that the restriction $\sF_{|C}$ is an ample vector bundle on $C$.
Then the leaf of $\sF$ through any point of $C$ is an algebraic variety.
\end{thm}

\section{Movable complete intersection curves and semistable sheaves}

In order to prove our results, we 
construct subfoliations of the foliation $\sF$ which 
inherit some of the positivity properties of $\sF$.
One way to construct such subfoliations
is via Harder-Narasimhan filtrations
(see Proposition \ref{proposition:HN_on_the_base} below).

The notion of slope-stability with respect to the choice of an ample line bundle is not flexible enough to allow for applications in birational geometry.
The paper \cite{gkp_movable} (see also \cite{campana_peternell11}) extends a number of known results from the classical case
to the setting where the ambient variety is normal and 
$\mathbb{Q}$-factorial, and polarisations are given by movable curve classes.
Recall that a numerical curve class $\alpha$ on a normal projective variety
is movable if 
$D\cdot \alpha \ge 0$ for all effective Cartier divisors $D$.
In this paper, it is advantageous to generalize the notion of slope, replacing movable curve classes with movable complete intersection curve classes. 

Let $X$ be an $n$-dimensional normal projective variety. 
Consider the space $\N_1(X)_\mathbb{R}$ of numerical curve classes on $X$.

\begin{notation}
Let $X$ be an $n$-dimensional normal projective variety, and let $\alpha$ be a numerical curve class on $X$. We say 
that $\alpha$ is a \emph{complete intersection numerical curve class} if 
$\alpha = [D_1 \cdot \cdots \cdot D_{n-1}] \in \N_1(X)_\mathbb{R}$ for some Cartier divisors 
$D_1,\ldots,D_{n-1}$ on $X$.
\end{notation}

\begin{rem}\label{remark:cincc}
Let $D_1,\ldots,D_{n-1}$ be Cartier divisors on $X$. Suppose that $D_i$ is nef for all $i$. Then 
$\alpha :=[D_1 \cdot \cdots \cdot D_{n-1}]$ is a movable complete intersection numerical curve class. 
\end{rem}

\begin{defn}Let $X$ be a normal projective variety.
Let $\alpha$ be a complete intersection numerical curve class on $X$, 
and let $\sF$ be a torsion-free sheaf of positive rank $r$.  

We
define the slope of $\sF$ with respect to $\alpha$ to be
$\mu_{\alpha}(\sF)=\frac{\det \sF\cdot \alpha}{r}$.  

We
say that  $\sF$ is
\emph{$\alpha$-semistable} if for any  
subsheaf $\sE\neq 0$ of $\sF$ we have $\mu_{\alpha}(\sE)\leq\mu_{\alpha}(\sF)$.
\end{defn}

The same argument used in the proof of 
\cite[Corollary 2.26]{gkp_movable} shows that the following holds.

\begin{prop}
Given a torsion-free sheaf $\sF$ on a normal projective variety $X$, 
and a movable complete intersection numerical curve class $\alpha$,
there exists a unique filtration
of $\sF$ by saturated subsheaves
$$
0=\sF_0\subsetneq \sF_1\subsetneq \cdots\subsetneq \sF_k=\sF,
$$
with $\alpha$-semistable quotients $\cQ_i=\sF_i/\sF_{i-1}$, and
such that $\mu_{\alpha}(\cQ_1) > \mu_{\alpha}(\cQ_2) > \cdots >
\mu_{\alpha}(\cQ_k)$.
\end{prop}
This is called the \emph{Harder-Narasimhan filtration} of $\sF$. The sheaf
$\sF_1$ is called the maximally destabilizing subsheaf of $\sF$, and it satisfies
$$\mu_\alpha(\sF_1)=\mu_\alpha^{\max}(\sF):=\sup \{\mu_\alpha(\sE)\,|\, 
0\neq \sE \subseteq \sF\text{ a coherent subsheaf}\}.$$

We will need the following easy observations.

\begin{lemma}\label{lemma:semistability_reflexive hull}
Let $\sF$ be a torsion-free sheaf on a normal projective $X$, and let $\alpha$ be a complete intersection 
numerical curve class on $X$. Then $\sF$ is $\alpha$-semistable if and only if so is $\sF^{**}$. 
\end{lemma}

\begin{notation}\label{notation:numerical_pullback}
Let $\phi \colon X \to Y$ be any birational morphism of projective normal varieties, and let 
$\alpha = [D_1 \cdot \cdots \cdot D_{n-1}] \in \N_1(X)_\mathbb{R}$ be 
complete intersection numerical curve class.
Set $\phi^*\alpha:=[\phi^*D_1 \cdot \cdots \cdot \phi^*D_{n-1}] \in \N_1(X)_\mathbb{R}$.
We refer to it as the \emph{numerical pull-back of} $\alpha$.
\end{notation}

\begin{rem}
In the setup of Notation \ref{notation:numerical_pullback}, we have
$D\cdot \phi^*\alpha = \phi_*D\cdot \alpha$ 
for any Weil $\mathbb{Q}$-divisor $D$ on $X$
by the projection formula. It follows that this construction agrees with the usual numerical pull-back if $X$ and $Y$ are
$\mathbb{Q}$-factorial.
\end{rem}

\begin{rem}
The pull-back of any movable complete intersection numerical curve class is again a 
movable complete intersection numerical curve class.
\end{rem}

\begin{lemma}\label{lemma:pull-back_semistability_birational_morphism}
Let $\phi \colon X \to Y$ be any birational morphism of projective normal varieties, and let $\sF$ be a torsion-free sheaf on $X$. Let $\alpha$ be a complete intersection numerical curve class on $Y$. Then 
$\sF$ is $\phi^*\alpha$-semistable 
if and only if the torsion-free sheaf $\phi_*\sF$ is $\alpha$-semistable.
\end{lemma}

The proof of the next lemma is similar to that of \cite[Lemma 3.2.2]{HuyLehn}, and so we leave the details to the reader.

\begin{lemma}\label{lemma:pull-back_semistability_finite_morphism}
Let $\phi \colon X \to Y$ be a finite surjective morphism of projective normal varieties, and let $\sF$ be a torsion-free sheaf on $Y$. Let $\alpha$ be a complete intersection numerical curve class on $Y$. Then 
$\sF$ is $\alpha$-semistable 
if and only if $\phi^*\sF/\textup{Tor}(\phi^*\sF)$ is $\phi^*\alpha$-semistable.
\end{lemma}

We will make use of the following notation.

\begin{notation}
Let $X$ be an $n$-dimensional normal projective variety, and let $|H|$ be a basepoint-free complete linear system on $X$. For each $1 \le i \le n-1$, let $D_i \in |H|$ be a general hypersurface, and set 
$C=D_1\cap \cdots \cap D_{n-1}$. We say that $C$ is a \emph{general complete intersection curve on $X$ of type
$H$}.
\end{notation}

We now provide a technical tool for the proof of the main results.
The following result is a generalization of \cite[Lemma 4.7]{fano_fols} (see also \cite[Lemma 12]{druel_paris}).

\begin{prop}\label{proposition:HN_on_the_base}
Let $\psi\colon Z \to Y$ be a dominant morphism of normal projective varieties,
let $X$ be a normal projective variety, and let
$\nu\colon Z \to X$ be a birational morphism.
Let $H$ be a very ample divisor on $X$, and let $\sG$ be a torsion-free sheaf on $Y$.
Then there exists a saturated subsheaf $\sE \subseteq \sG$ satisfying the following property.
For a general complete intersection curve $C$ on $Z$ of type
$m\nu^*H$ with $m$ large enough, ${\psi^*\sE}_{|C}$ is the maximally destabilizing subsheaf of $\psi^*\sG_{|C}$.
\end{prop}

\begin{proof}Set $n=\dim  Z=\dim X$ and $l=\dim Y$. 

\medskip

Let $m$ be a positive integer, and
consider general hypersurfaces $H_i\in|mH|$ for $1 \le i\le n-1$. 
By replacing $m$ with a larger integer, we may assume that
the Harder-Narasimhan filtration of $\sF:=\big(\nu_*(\psi^*\sG)\big)^{**}$ with respect to $H$ commutes with
restriction to all complete intersections $H_1\cap\cdots\cap H_i$, $1\le i\le n-1$ 
(see \cite{flenner_restriction}).

Set $D_i:=\nu^{-1}(H_i) \in|m\nu^*H|$, $Z_1:=D_1\cap \cdots \cap D_{n-l}$, and 
$X_1:=H_1\cap\cdots\cap H_{n-l}$.
The restriction of $\nu$ to $Z_1$ yields a surjective birational morphism
$\nu_1\colon Z_1 \to X_1$ of normal varieties. 
Denote by 
$\psi_1\colon Z_1 \to Y$ the restriction of $\psi$ to $Z_1$. Note that $\psi_1$ is generically finite. Let $\psi_2\colon Z_1 \to Y_1$ be its Stein factorization. 
It comes with a birational morphism $\mu_1\colon Y_1 \to Y$.

Set $B=H_1\cap\cdots\cap H_{n-1}$. Up to replacing $H_{n-l+1},\ldots,H_{n-1}$ by linearly equivalent divisors on $X$, we may assume that $B \subset X \setminus \nu_1(\textup{Exc}(\nu_1)\big)$.
Set $C_1=\nu_1^{-1}(B)=\nu^{-1}(B)=D_1\cap \cdots \cap D_{n-1} \cong B$. 

Set $\sG_1:=\mu_1^*\sG/\textup{Tor}(\mu_1^*\sG)$,
$\sT:=(\psi^*\sG)^{**}$, $\sT_1=(\psi_1^*\sG)^{**}\cong (\psi_2^*\sG_1)^{**}$, and
$\sF_1:=\big((\nu_1)_*\sT_1\big)^{**}$.
Since 
$\sF_{|X_1}$ is reflexive, we must have
$$
\sF_1\cong \sF_{|X_1}.$$

Let $\sK$ be the maximally destabilizing subsheaf of $\sF$ with respect to $H$.
By Flenner's version of the Mehta-Ramanathan theorem and the choice of $m$, 
$\sK_1:=\sK_{|X_1}$ is the maximally destabilizing subsheaf of $\sF_1\cong \sF_{|X_1}$
with respect to $H_{|X_1}$.
Note that $\sK_1$ is reflexive by \cite[Remark 2.3]{fano_fols}. Let $\sS_1 \subset \sT_1$
such that $\big((\nu_1)_*\sS_1\big)^{**}=\sK_1 \subset \sF_1$. We may assume that 
$\sS_1$ is saturated in $\sT_1$.
Since $B \subset X \setminus \nu_1(\textup{Exc}(\nu_1)\big)$, we have 
$$\sK_{|B}={\sK_1}_{|B}={\big((\nu_1)_*\sS_1\big)^{**}}_{|B}\cong {(\nu_1)_*\sS_1}_{|B}\cong {\sS_1}_{|C_1}$$
and
$$\sF_{|B}={\sF_1}_{|B}={\big((\nu_1)_*\sT_1\big)^{**}}_{|B}\cong{(\nu_1)_*\sT_1}_{|B}\cong {\sT_1}_{|C_1}.$$
By Flenner's version of the Mehta-Ramanathan theorem and the choice of $m$,
we conclude as above that 
${\sS_1}_{|C_1}$ is the maximally destabilizing subsheaf of ${\sT_1}_{|C_1}$.

Let $K$ be a splitting field of the function field $\mathbb{C}(Z_1)$ over $\mathbb{C}(Y_1)$, and let 
$\nu_2\colon Z_2 \to Z_1$ be 
the normalization of $Z_1$ in $K$. 
Set $\psi_3:=\psi_2\circ\nu_2\colon Z_2 \to Y_1$, and denote by $G$ the Galois group of 
$\mathbb{C}(Z_2)$ over $\mathbb{C}(Y_1)$. We obtain a commutative diagram

\centerline{
\xymatrix{
Z_2\ar[rr]^{\nu_2,\text{ finite}}\ar[d]^{\psi_3,\text{ finite}} && Z_1 \ar@{=}[rr]\ar[d]^{\psi_2, \text{ finite}} && Z_1\ar@{^{(}->}[rrd] \ar[rr]^{\nu_1,\text{ birational}}\ar[d]^{\psi_1} &&  X_1\ar@{^{(}->}[drr] &&\\
Y_1 \ar@{=}[rr] && Y_1 \ar[rr]_{\mu_1,\text{ birational}} && Y \ar@{=}[rrd] && Z \ar[d]^{\psi}\ar[rr]_{\nu,\text{ birational}} && X \\
&&&&&& Y. && \\
}
}

The numerical curve class  $[C_1]\in\N_1(Z_1)_\mathbb{R}$
is a movable complete intersection curve class by Remark \ref{remark:cincc}.
By Lemmas \ref{lemma:semistability_reflexive hull} and \ref{lemma:pull-back_semistability_birational_morphism},
$\sS_1$ is the maximally destabilizing sheaf of 
$\sT_1$ with respect to the movable class $[C_1]=[\nu_1^*H_{|X_1}]^{l -1}\in \N_1(Z_1)_{\mathbb{R}}$.
From Lemmas \ref{lemma:semistability_reflexive hull} and \ref{lemma:pull-back_semistability_finite_morphism}, we deduce that
$\sS_2:=(\nu_2^*\sS_1)^{**}$ is the maximally destabilizing sheaf of 
$(\nu_2^*\sT_1)^{**}\cong (\psi_3^*\sG_1)^{**}=:\sT_2$ with respect to the movable class 
$\nu_2^*[C_1]=[(\nu_1\circ\nu_2)^*H_{|X_1}]^{l -1}\in \N_1(Z_2)_{\mathbb{R}}$.

Let $Y_1^\circ \subset Y_1$ be a dense open subset 
with $\codim\, Y_1 \setminus Y_1^\circ \ge 2$ 
such that $G$ acts on 
$Z_2^\circ:=\psi_3^{-1}(Y_1^\circ)$ and such that ${\sG_1}_{|Y_1^\circ}$ is locally free. Note that 
$Z_2^\circ/ G \cong Y_1^\circ$.
Because of its uniqueness, ${\sS_2}_{|Z_2^\circ}$ 
is invariant under the natural action of $G$ on ${\sT_2}_{|Z_2^\circ}\cong{\psi_3^*\sG_1}_{|Z_2^\circ} $. Therefore there exists a saturated subsheaf
$\sE_1 \subseteq \sG_1$ such that ${\psi_3^*\sE_1}_{|Z_2^\circ}\cong{\sS_2}_{|Z_2^\circ}$. By
Lemma \ref{lemma:pull-back_semistability_finite_morphism} again, we have
that $(\psi_2^*\sE_1)^{**}\cong\sS_1$.

Consider $\sE = (\mu_1)_*\sE_1 \subseteq \sG$. Then $\sE$ is saturated in $\sG$, and 
${\psi^*\sE}_{|C}\cong {\sS_1}_{|C_1}$ is the maximally destabilizing subsheaf of 
$\psi^*\sG_{|C}\cong {\sT_1}_{|C_1}$. The completes the proof of the proposition.
\end{proof}

\section{The anti-canonical divisor of an algebraically integrable foliation}

In this section, we apply Viehweg's weak positivity theorem to algebraically integrable foliations. We refer to
\cite{viehweg_moduli_polarized_manifolds} for the definition of weak positivity. This notion was introduced by Viehweg, as a kind of birational version of being nef.

The following is a generalization of \cite[Lemma 2.14]{hoering} 
and \cite[Theorem 2.11]{campana_paun}
(see also Proposition 
\ref{prop:pseudo_effective2} below). Recall that a $\mathbb{Q}$-divisor $D$ on a normal projective variety $X$ is said to be \textit{pseudo-effective} if, for any big $\mathbb{Q}$-divisor $B$ on $X$ and any rational number 
$\varepsilon > 0$, there exists an effective $\mathbb{Q}$-divisor $E$ on $X$ such that
$D + \varepsilon B \sim_\mathbb{Q} E$.

\begin{prop}\label{prop:pseudo_effective}
Let $X$ be a normal projective variety, let $\sH$ be a $\mathbb{Q}$-Gorenstein algebraically integrable 
(possibly singular)
foliation on $X$, and let $(F,\Delta)$ be its general log leaf. Let $L$ be an effective $\mathbb{Q}$-Cartier $\mathbb{Q}$-divisor on $X$. Suppose that $(F,\Delta+L_{|F})$ is log canonical, and that
$\kappa\big(F,K_F+\Delta+L_{|F}\big) \ge 0$. 
Then $K_\sH+L$ is pseudo-effective.
\end{prop}

\begin{proof}
Let $\psi\colon Z\to Y$ be the family of leaves, and let $\nu\colon Z \to X$ be the natural morphism (see \ref{family_leaves}).
By Lemma \ref{lemma:reduced_fibers_codimension_1} below,
there exists a finite surjective morphism
$Y_1 \to Y$ with $Y_1$ normal and connected satisfying the following property. If $Z_1$ denotes the normalization of the product $Y_1 \times_Y Z$, then the induced morphism $\psi_1\colon Z_1 \to Y_1$ has reduced fibers over codimension one points in $Y_1$. 

By \ref{family_leaves},
there is a canonically defined effective $\mathbb{Q}$-Weil divisor $B$ on $Z$ such that 
\begin{equation}\label{eq:eq1}
K_{Z/Y}-R(\psi)+B\sim_\bQ \nu^* K_\sH
\end{equation}
where $R(\psi)$ denotes the ramification divisor of $\psi$. 
A straightforward computation shows that 
\begin{equation}\label{eq:eq2}
\nu_1^*\big(K_{Z/Y}-R(\psi)\big)=K_{Z_1/Y_1}.
\end{equation}
Set $\Gamma_1=\nu_1^*B+(\nu\circ\nu_1)^*L$.
Let $F_1$ be a general fiber of $\psi_1$. Note that 
$(F_1,{\Gamma_1}_{|F_1})$ has log canonical singularities by assumption. It follows that 
$(Z_1,\Gamma_1)$ has log canonical singularities over the generic point of $Y_1$ by inversion of adjunction
(see \cite{kawakita}). Let $\mu_2\colon Y_2\to Y_1$ be a resolution of singularities, and let 
$Z_2$ be a resolution of singularities of the product $Y_2 \times_{Y_1} Z_1$.
Up to replacing $Z_2$ with a birational model, we may assume that 
$Z_2$ is a log resolution of $(Z_1,\Gamma_1)$. 
We have a commutative diagram
\centerline{
\xymatrix{
Z_2 \ar[rr]^{\nu_2,\text{ birational}}\ar[d]^{\psi_2} && 
Z_1 \ar[rr]^{\nu_1,\text{ finite}}\ar[d]^{\psi_1} && Z \ar[d]^{\psi}\ar[rr]^{\nu,\text{ birational}} && X \\
Y_2 \ar[rr]_{\mu_2,\text{ birational}} && Y_1 \ar[rr]_{\mu_1,\text{ finite}} && Y. && \\
}
}
We write 
\begin{equation}\label{eq:log}
K_{Z_2}+\Gamma_2=\nu_2^*(K_{Z_1}+\Gamma_1)+E
\end{equation}
where $\Gamma_2$ and $E$ are effective, with no common components, $(\nu_2)_*\Gamma_2=\Gamma_1$, and $E$ is $\nu_2$-exceptional. Then $(Z_2,\Gamma_2)$ has log canonical singularities over the generic point of $Y_2$. Let $F_2$ denotes a general fiber of $\psi_2$.
Then
$$
\begin{array}{ccll}
\kappa\big(F_2,K_{F_2}+{\Gamma_2}_{|F_2}\big) & = & \kappa\big(F_1,K_{F_1}+{\Gamma_1}_{|F_1}\big) & \text{ by } \eqref{eq:log}\\
& = & \kappa\big(F,K_{F}+{B}_{|F}+{L}_{|F}\big) 
& \text{ by } \eqref{eq:eq2}\\
& =  & \kappa\big(F,K_{F}+\Delta+{L}_{|F}\big) 
 & \text{ by } \ref{family_leaves}\\
& \ge  & 0.  & \\
\end{array}
$$
Thus, for any positive integer $m$ which is sufficiently divisible, the natural morphism 
$$\psi_2^*\Big((\psi_2)_*\sO_{Z_2}\big(m(K_{Z_2/Y_2}+\Gamma_2)\big)\Big)\to  \sO_{Z_2}\big(m(K_{Z_2/Y_2}+\Gamma_2)\big)$$
is generically surjective. 
The sheaf 
$(\psi_2)_*\sO_{Z_2}\big(m(K_{Z_2/Y_2}+\Gamma_2)\big)$ is weakly positive
by \cite[Theorem 4.13]{campana04}. Therefore, $K_{Z_2/Y_2}+\Gamma_2$ is pseudo-effective, and hence
so is $(\nu\circ\nu_1\circ\nu_2)_*(K_{Z_2/Y_2}+\Gamma_2)$. Since
$\psi_2$ is equidimensional, there exist dense open subsets
$Y_1^\circ\subset Y_1$ and $Z_1^\circ \subset Z_1$ with 
$\codim \, Y_1 \setminus Y_1^\circ \ge 2$ and $\codim \, Z_1 \setminus Z_1^\circ \ge 2$
such that $\psi_1(Z_1 \setminus Z_1^\circ)\subset Y_1 \setminus Y_1^\circ$, and such that
$\nu_2$ (respectively, $\mu_2$) induces an isomorphism
$\nu_2^{-1}(Z_1 \setminus Z_1^\circ) \cong Z_1 \setminus Z_1^\circ$
(respectively, $\mu_2^{-1}(Y_1 \setminus Y_1^\circ) \cong Y_1 \setminus Y_1^\circ$).
Hence, $(\nu_2)_*(K_{Z_2/Y_2}+\Gamma_2)=K_{Z_1/Y_1}+\Gamma_1$.
From \eqref{eq:eq1} and \eqref{eq:eq2}, we conclude that 
$$(\nu\circ\nu_1\circ\nu_2)_*(K_{Z_2/Y_2}+\Gamma_2)=\deg(\nu_1)(K_\sH+L)$$
is pseudo-effective, completing the proof of the proposition.
\end{proof}

\begin{lemma}\label{lemma:reduced_fibers_codimension_1}
Let $\psi\colon Z \to Y$ be a dominant equidimensional morphism of normal varieties. There exists a finite surjective morphism
$Y_1 \to Y$ with $Y_1$ normal and connected satisfying the following property. If $Z_1$ denotes the normalization of the product $Y_1 \times_Y Z$, then the induced morphism $\psi_1\colon Z_1 \to Y_1$ has reduced fibers over codimension one points in $Y_1$. 
\end{lemma}

\begin{proof}
This follows easily from \cite[Theorem 2.1']{bosch95} (see also \cite[Section 5]{abramovich_karu}).  
\end{proof}

As a first application of Proposition \ref{prop:pseudo_effective}, we obtain
a generalization of \cite[Lemma 2.14]{hoering} and \cite[Theorem 2.11]{campana_paun}.

\begin{prop}\label{prop:pseudo_effective2}
Let $\phi\colon X \to Y$ be a dominant morphism of normal projective varieties with general fiber $F$.
Suppose that $Y$ is smooth, and that
$\phi$ has connected fibers.
Let $L$ be an effective $\mathbb{Q}$-divisor on $X$
such that $K_X+L$ is $\mathbb{Q}$-Cartier in a neighborhood of $F$.
Suppose furthermore that $(F,L_{|F})$ is log canonical, and that
$\kappa\big(F,K_F+L_{|F}\big) \ge 0$. 
Then $K_{X/Y}-R(\phi)+L$ is pseudo-effective, where $R(\phi)$ denotes the ramification divisor of $\phi$.
\end{prop}

\begin{proof}
Let $\nu\colon X_1 \to X$ be a log resolution of $(X,L)$, and set $\phi_1:=\phi\circ\nu\colon X_1 \to Y$.
Let $Y^\circ\subset Y$ be a dense open subset such that $K_{X^\circ}+L_{|X^\circ}$
is $\mathbb{Q}$-Cartier, where $X^\circ:=\phi^{-1}(Y^\circ)$.
Set $X_1^\circ:=\nu^{-1}(X^\circ)$.
We write 
\begin{equation}\label{eq:log2}
K_{X_1^\circ}+L_1^\circ=(\nu_{|X_1^\circ})^*(K_{X^\circ}+{L_{|X^\circ}})+E^\circ
\end{equation}
where $L_1^\circ$ and $E^\circ$ are effective, with no common components, 
${(\nu_{|X_1^\circ})_*}L_1^\circ=L_{|X^\circ}$, and $E^\circ$ is $\nu$-exceptional. 
Denote by $L_1$ the Zariski closure of $L_1^\circ$ in $X_1$. Note that
$\nu_*L_1=L$.
Let $F_1$ denotes a general fiber of $\phi_1$.
Then 
$(F_1,{L_1}_{|F_1})$ has log canonical singularities, 
and $\kappa\big(F_1,K_{F_1}+{L_1}_{|F_1}\big)=\kappa\big(F,K_{F}+{L}_{|F}\big) \ge 0$
by \eqref{eq:log2}.
By Proposition \ref{prop:pseudo_effective} applied to the foliation 
$\sH_1$ on $X_1$ induced by $\phi_1$ and $L=L_1$, we conclude that 
$K_{\sH_1}+L_1$ is pseudo-effective.
There is an exact sequence
$$0 \to \sH_1 \to T_{X_1} \to \phi_1^*T_Y,$$
and thus $K_{\sH_1}=K_{X_1/Y}-R(\phi_1)-G$, where 
$R(\phi_1)$ denotes the ramification divisor of $\phi_1$ 
and $G$ is an effective divisor.
This in turn implies that $K_{X/Y}-R(\phi)+L$ is pseudo-effective 
since $K_{X/Y}-R(\phi) -\nu_*\big(K_{X_1/Y}-R(\phi_1)-G\big)=\nu_*G$ is effective.
\end{proof}

\begin{rem}
It should be noted that Proposition \ref{prop:pseudo_effective2} does not require 
$K_{X/Y}-R(\phi)+L$ to be $\mathbb{Q}$-Cartier. 
\end{rem}

\begin{cor}\label{cor:pseudo_effective2}
Let $\psi\colon Z \to Y$ be a dominant morphism of normal projective varieties with general fiber $F$.
Suppose that $\psi$ is equidimensional with connected fibers.
Let $L$ be an effective $\mathbb{Q}$-divisor on $Z$
such that $K_Z+L$ is $\mathbb{Q}$-Cartier in a neighborhood of $F$.
Suppose furthermore that $(F,L_{|F})$ is log canonical, and that
$\kappa\big(F,K_F+L_{|F}\big) \ge 0$. 
Then $K_{Z/Y}-R(\psi)+L$ is pseudo-effective, where $R(\psi)$ denotes the ramification divisor of $\psi$.
\end{cor}

\begin{proof}
Let $\mu_1\colon Y_1 \to Y$ be a resolution of singularities, and let $Z_1$ be the normalization of the
product $Y_1 \times_Y Z$, with natural morphisms $\psi_1\colon Z_1 \to Y_1$ and $\nu_1\colon Z_1\to Z$.
Apply Proposition \ref{prop:pseudo_effective2} to $\psi_1$ and $\nu_1^*L$. We conclude that 
$K_{Z_1/Y_1}-R(\psi_1)+\nu_1^*L$ is pseudo-effective, where $R(\psi_1)$ denotes the ramification divisor of $\psi_1$. It follows that $(\nu_1)_*(K_{Z_1/Y_1}-R(\psi_1)+\nu_1^*L)$ is pseudo-effective as well.
Since
$\psi$ is equidimensional, there exist dense open subsets
$Y^\circ\subset Y$ and $Z^\circ \subset Z$ with 
$\codim \, Y \setminus Y^\circ \ge 2$ and $\codim \, Z \setminus Z^\circ \ge 2$
such that $\psi(Z \setminus Z^\circ)\subset Y \setminus Y^\circ$, and such that
$\nu_1$ (respectively, $\mu_1$) induces an isomorphism
$\nu_1^{-1}(Z \setminus Z^\circ) \cong Z \setminus Z^\circ$
(respectively, $\mu_1^{-1}(Y \setminus Y^\circ) \cong Y \setminus Y^\circ$).
Hence, $(\nu_1)_*\big(K_{Z_1/Y_1}-R(\psi_1)+\nu_1^*L\big)=K_{Z/Y}-R(\psi)+L$, where $R(\psi)$ denotes the ramification divisor of $\psi$. This completes the proof of the Corollary \ref{cor:pseudo_effective2}.
\end{proof}

Next, we obtain a generalization
of \cite[Proposition 5.8]{fano_fols}.

\begin{prop}\label{proposition:thm_alg_int}
Let $X$ be a normal projective variety, let $\sH$ a $\mathbb{Q}$-Gorenstein algebraically integrable 
foliation on $X$, and let $(F,\Delta)$ be its general log leaf. 
Suppose that $-K_\sH=A+E$ where 
$A$ is a $\mathbb{Q}$-ample $\mathbb{Q}$-divisor and
$E$ is an effective $\mathbb{Q}$-divisor.
Then $(F,\Delta+E_{|F})$ is not klt.
\end{prop}

\begin{proof}We argue by contradiction, and assume that $(F,\Delta+E_{|F})$ is klt.
Let $\psi\colon Z\to Y$ be the family of leaves, and let $\nu\colon Z \to X$ be the natural morphism (see \ref{family_leaves}). There is a canonically defined effective $\mathbb{Q}$-Weil divisor $B$ on $Z$ such that 
$$
K_{Z/Y}-R(\psi)+B\sim_\bQ \nu^* K_\sH
$$
where $R(\psi)$ denotes the ramification divisor of $\psi$. 
We view $F$ as a (general) fiber of $\psi$.
Recall that $B_{|F}=\Delta$, and that 
$K_{Z/Y}-R(\psi)$ is a canonical divisor of the foliation $\sH_Z$ on $Z$ induced by $\sH$ (or $\psi$).
We write $\nu^*A=H+N$ where $H$ is a $\mathbb{Q}$-ample $\mathbb{Q}$-divisor and $N$ is an effective 
$\mathbb{Q}$-divisor. Let $0<\varepsilon \ll 1$ be a rational number such that 
$(F,\Delta+E_{|F}+\varepsilon N_{|F})$ is klt, and set $H_\varepsilon=\nu^*A-\varepsilon N$.
Note that $H_\varepsilon$ is $\mathbb{Q}$-ample.
Finally, let $D\neq 0$ be an effective $\mathbb{Q}$-Cartier $\mathbb{Q}$-divisor on $Y$, such that 
$H'_\varepsilon:=H_\varepsilon-\psi^*D$ is $\mathbb{Q}$-ample. 
Then $K_F+\Delta+E_{|F}+{H'_\varepsilon}_{|F}+\varepsilon N_{|F}\sim_\bQ 0$.
Now, apply Proposition \ref{prop:pseudo_effective} to conclude 
that $K_{\sH_Z}+B+\nu^*E+H'_\varepsilon+\varepsilon N$ is pseudo-effective.
But 
$$K_{\sH_Z}+B+\nu^*E+H'_\varepsilon+\varepsilon N\sim_\mathbb{Q}-\psi^*D,$$
yielding a contradcition.
\end{proof}

\begin{rem}Note that Proposition \ref{proposition:thm_alg_int} also follows Proposition \ref{proposition:lc_center}. We decided to include a proof to keep our exposition as self-contained as possible.
\end{rem}

\section{Singularities of foliations}

In this section, we provide another application of Proposition \ref{prop:pseudo_effective}: we study singularities of algebraically integrable foliations (see Proposition \ref{prop:singularities_foliations}). First, we recall the definition of terminal and canonical singularities, inspired by the theory of singularities of pairs, developed in the context of the minimal model program.

\begin{say}[{\cite[Definition I.1.2]{mcquillan08}}]
Let $\sF$ be a foliation on a normal variety $X$.
Given a birational morphism $\nu_1\colon X_1 \to X$ of normal varieties,
there is a unique foliation $\sF_1$ on $X_1$ that agrees with $\nu_1^*\sF$ on the open 
subset of $X_1$ where $\nu_1$ is an isomorphism. Let $B=\sum_{i\in I}a_i B_i$ be a 
(not necessarily effective) 
$\mathbb{Q}$-divisor on $X$.
Suppose that $B_i$ is not invariant by $\sF$ for any $i\in I$. 

Suppose moreover that $K_\sF+B$ is $\bQ$-Cartier and that $\nu_1$ is projective.
Then there are uniquely defined rational numbers $a_E(X,\sF,B)$'s such that
$$
K_{\sF_1}+B_1=\nu_1^*(K_{\sF}+B)+ \sum_E a_E(X,\sF,B)E,
$$
where $E$ runs through all exceptional prime divisors for $\nu_1$, and where $B_1$ denotes the proper transform of $B$ in $X_1$.
The $a_E(X,\sF,B)$'s do not depend on the birational morphism $\nu_1$,
but only on the valuations associated to the $E$'s.

We say that $(\sF,B)$ is \emph{terminal} (respectively, \emph{canonical}) if
$a_E(X,\sF,B) > 0$ (respectively, $a_E(X,\sF,B)\ge 0$) 
for all $E$ exceptional over $X$.

We say that $\sF$ is \emph{terminal} (respectively, \emph{canoncial}) if so is $(\sF,0)$.
\end{say}

\begin{rem}\label{remark:regular_canonical}
A regular foliation on a smooth variety is canonical by \cite[Lemma 3.10]{fano_fols}. 
\end{rem}

The following example shows that a regular foliation on a smooth variety may not be (log) terminal.

\begin{exmp}
Let $F$ and $T$ be smooth varieties with $\dim F \ge 1$, and $\dim T \ge 2$. Let $\sF$ be the foliation 
on $X:=F\times T$ induced by the projection $X \to T$. Fix a point $t\in T$, and let 
$T_1$ be the blow-up of $T$ at $t$. Set $X_1:=F \times T_1$, with natural morphism $\nu_1\colon X_1 \to X$. 
Note that
$\nu_1\colon X_1 \to X$ is the blow-up of $X$ along $F \times \{t\} \subset F \times T$.
Denote by $E$ its exceptional set. Observe that $E$ is invariant by $\sF_1$.
The foliation $\sF_1$ induced by $\sF$
on $X_1$ is given by the natural morphism $X_1 \to T_1$. We have 
$\sF_1 = \nu_1^*\sF$, and hence $a_E(X,\sF,0)=0$. 
\end{exmp}

The following result is probably well-known to experts, though as far as we can see, the statement does not appear in the literature.

\begin{lemma}\label{lemma:canonical_singularities}
Let $\psi\colon Z \to Y$ be a dominant equidimensional morphism of normal varieties, and 
denote by $\sH$ the foliation on $Z$ induced by
$\psi$. Suppose that $Y$ is smooth. 
Let $B=\sum_{i\in I}a_i B_i$ be a $\mathbb{Q}$-Cartier $\mathbb{Q}$-divisor on $Z$. Suppose that 
$B_i$ is not invariant by $\sH$ for every $i\in I$. Suppose furthermore that
$(\sH,B)$ is terminal (respectively, canonical). Then $\big(Z,-R(\psi)+B\big)$ has terminal singularities (respectively, canonical singularities).
\end{lemma}

\begin{proof}
By \ref{pullback_foliations}, we have $K_\sH=K_{Z/Y}-R(\psi)$, and thus
$K_{Z}-R(\psi)$ is $\mathbb{Q}$-Cartier.
Let $\nu_1\colon Z_1 \to Z$ be a resolution of singularities, and set 
$\psi_1:=\psi \circ \nu_1$. Denote by
$\sH_1$ the foliation induced by $\sH$ (or $\psi_1$) on $Z_1$. Write
$K_{\sH_1}+B_1=\nu^*(K_\sH+B)+E$ where $E$ is a $\nu$-exceptional $\mathbb{Q}$-divisor, and $B_1$ denotes the proper transform of $B$ in $Z_1$. There is an exact sequence
$$0 \to \sH_1 \to T_{Z_1} \to \psi_1^*T_Y,$$
and thus $K_{\sH_1}=K_{Z_1/Y}-R(\psi)_{1}-G$, where 
$R(\psi)_{1}$  denotes the proper transform of $R(\psi)$ in $Z_1$ 
and $G$ is an effective $\nu$-exceptional $\mathbb{Q}$-divisor.
We obtain $$K_{Z_1}-R(\psi)_{1}+B_1=\nu_1^*\big(K_Z-R(\psi)+B\big)+E+G.$$
The lemma follows easily.
\end{proof}

\begin{prop}\label{prop:singularities_foliations}
Let $X$ be a normal projective variety, let $\sH$ be a $\mathbb{Q}$-Gorenstein algebraically integrable 
foliation on $X$, and let $(F,\Delta)$ be its general log leaf. Suppose that $\Delta=0$.
Suppose moreover that
$-K_\sH \sim_\mathbb{Q} P+N$ where 
$P$ is a nef $\mathbb{Q}$-divisor and
$N=\sum_{i\in I}N_i$ is an effective $\mathbb{Q}$-divisor such that $(F,N_{|F})$ is canonical. Suppose furthermore that $N_i$ is not invariant by $\sH$ for any $i\in I$.
Then $(\sH,N)$ is canonical.
\end{prop}

\begin{proof}
Let $\psi\colon Z\to Y$ be the family of leaves, and let $\nu\colon Z \to X$ be the natural morphism (see \ref{family_leaves}).
Let $\nu_1 \colon Z_1 \to Z$ be a birational morphism with $Z_1$ smooth and projective. We obtain a commutative diagram

\centerline{
\xymatrix{
Z_1 \ar[rr]^{\nu_1,\text{ birational}}\ar[d]^{\psi_1} && Z \ar[d]^{\psi}\ar[rr]^{\nu,\text{ birational}} && X \\
Y \ar@{=}[rr] && Y. && \\
}
}
\noindent Denote by $\sH_{Z_1}$ the foliation on $Z_1$ induced by 
$\sH$.
Let $E_1$ be the $\nu\circ\nu_1$-exceptional
$\mathbb{Q}$-divisor such that
\begin{equation}\label{eq:canonical_bundle_formula}
K_{\sH_{Z_1}}+N_1=(\nu\circ\nu_1)^*(K_\sH+N)+E_1,
\end{equation}
where $N_1$ denotes the proper transform of $N$ in $Z_1$. To prove the lemma, it is enough to show that $E_1$ is effective.
Denote by $F_1$ the proper transform of $F$ in $Z_1$. 
Note that we have ${K_\sH}_{|F}\sim_\mathbb{Q}K_F$ since $\Delta=0$.
Since $(F,N_{|F})$ is canonical, the restriction ${E_1}_{|F}$ of $E_1$ to $F$
is effective, and $(F_1,{N_1}_{|F_1})$ is canonical as well (see \cite[Lemma 3.10]{kollar97}). 
Let $H$ be an ample (effective) divisor on $Z_1$, let $\varepsilon > 0$ be a rational number, and
let $A_\varepsilon\sim_\mathbb{Q}(\nu\circ\nu_1)^*P+\varepsilon H$
such that $(F_1,{N_1}_{|F_1}+{A_\epsilon}_{|F_1})$ is canonical.
We have
$$K_{F_1}+{N_1}_{|F_1}=({\nu_1}_{|F_1})^*\big(K_F+N_{|F}\big)+{E_1}_{|F_1}
\sim_\mathbb{Q}-({\nu_1}_{|F_1})^*(P_{|F})+{E_1}_{|F_1}$$
by \eqref{eq:canonical_bundle_formula} using
\ref{pullback_foliations} and the fact that $\Delta=0$, and hence 
$$\kappa\big(F_1,K_{F_1}+{N_1}_{|F_1}+{A_\varepsilon}_{|F_1}\big)
=\kappa\big(F_1,{E_1}_{|F_1}+\varepsilon H_{|F_1}\big)=\dim F_1
\ge 0.$$
By Proposition \ref{prop:pseudo_effective} applied to $\sH_{Z_1}$
and $L=A_\varepsilon+N_1\sim_\mathbb{Q}(\nu\circ\nu_1)^*P+\varepsilon H + N_1$ , 
we conclude that 
$$K_{\sH_{Z_1}}+A_\varepsilon+N_1\sim_\mathbb{Q}
K_{\sH_{Z_1}}+(\nu\circ\nu_1)^*P+\varepsilon H + N_1
$$
is pseudo-effective for any $0<\varepsilon \ll 1$, and hence
$K_{\sH_{Z_1}}+(\nu\circ\nu_1)^*P+N_1$ is pseudo-effective as well.
Also, by \eqref{eq:canonical_bundle_formula}, we have
$$
K_{\sH_{Z_1}}+(\nu\circ\nu_1)^*P+N_1=E_1.
$$
By a result of Lazarsfeld (\cite[Corollary 13]{kollar_larsen}), the 
$\nu\circ\nu_1$-exceptional $\mathbb{Q}$-divisor $E_1$ is pseudo-effective
if and only if it is effective, completing the proof of the proposition.
\end{proof}

The proof of the next Proposition is similar to that of Proposition \ref{prop:singularities_foliations}. One only needs to replace the use of Proposition \ref{prop:pseudo_effective} with Proposition \ref{prop:pseudo_effective2}.

\begin{prop}\label{prop:singularities_foliations2}
Let $\phi\colon X \to Y$ be a dominant morphism of normal projective varieties with general fiber $F$.
Suppose that $Y$ is smooth, and that
$\phi$ has connected fibers.
Suppose moreover that
$-K_{X/Y} \sim_\mathbb{Q} P+N$ where 
$P$ is a nef $\mathbb{Q}$-divisor and
$N$ is an effective $\mathbb{Q}$-divisor such that $(F,N_{|F})$ is canonical. 
Then $(X,N)$ has canonical singularities.
\end{prop}

\section{Foliations with nef anti-canonical bundle}

In this section, we provide another technical tool for the proof of the main results. The following result is the main observation of this paper. Note that Proposition \ref{proposition:regular_versus_algebraically_integrable_foliation_intro} is an immediate consequence of 
Proposition \ref{proposition:regular_versus_algebraically_integrable_foliation} 
and Lemma \ref{lemma:compact_leaf}
below.

\begin{prop}\label{proposition:regular_versus_algebraically_integrable_foliation}
Let $X$ be a complex projective manifold, and let $\sF\subset T_X$ be a foliation. 
Suppose that 
$-K_\sF \equiv P+N$ where 
$P$ is a nef $\mathbb{Q}$-divisor and
$N$ is an effective $\mathbb{Q}$-divisor such that
$(X,N)$ is log canonical.
Suppose that the algebraic part $\sH$ of $\sF$ has a compact leaf. Let $\psi\colon Z\to Y$ be the family of leaves of $\sH$, and let $\nu\colon Z \to X$ be the natural morphism. Set 
$\phi:=\psi \circ \nu^{-1}\colon X \dashrightarrow Y$, let
$\sG$ be the foliation on $Y$ such that $\sF=\phi^{-1}\sG$, and let $\sH_Z$ be the foliation on $Z$ induced by $\sH$. Then

\begin{enumerate}
\item $\phi^*K_\sG\equiv 0$,
\item $K_\sH\equiv K_\sF$, and
\item $K_{\sH_Z}\sim\nu^*K_\sH$.
\end{enumerate}
\end{prop}

\begin{proof}
We subdivide the proof into a number of relatively independent steps.  
\bigskip

\noindent\textbf{Step 1}. 
By assumption, $\phi$ is an almost proper map: there exist dense Zariski open sets 
$X^\circ \subset X$ and 
$Y^\circ \subset Y$ such that the the restriction $\phi^\circ$ of
$\phi$ to $X^\circ$ induces a proper morphism
$\phi^\circ\colon X^\circ \to Y^\circ$.
Moreover, the following holds: 
\begin{enumerate}
\item[$(\Diamond)$] there is no positive-dimensional algebraic subvariety passing through a general point of $Y$ that is tangent to 
$\sG$. 
\end{enumerate}

\bigskip

\noindent\textbf{Step 2}. Let $A$ be an ample divisor on $X$. 
Note that $F$ is smooth. There exists an open set $U \supset F$ such that 
$\sH_{|U}$ is a subbundle of $\sF_{|U}$, $\sG_{|U}$ is locally free,
and 
${\big(\sF/\sH\big)}_{|U}\cong {\phi^*\sG}_{|U}$.
In particular, ${K_\sF}_{|F}\sim {K_\sH}_{|F}\sim K_F$.
Let $\varepsilon > 0$ be a rational number. Then
$K_{F}-{K_\sF}_{|F}+\varepsilon A_{|F}\equiv \varepsilon A_{|F}$
is $\mathbb{Q}$-ample, and hence 
$\kappa(F,K_{F}+P_{|F}+N_{|F}+\varepsilon A_{|F}) \ge 0$. 
Since $P+\varepsilon A$ is $\mathbb{Q}$-ample, there exists
an effective $\mathbb{Q}$-divisor $A_\varepsilon \sim_\mathbb{Q}P+\varepsilon A$ such that
the pair $(X,N+A_\varepsilon)$ is log canonical.
From \cite[Theorem 4.8]{kollar97}, we conclude that
$(F,N_{|F}+{A_\varepsilon}_{|F})$ is log canonical. 
By Proposition \ref{prop:pseudo_effective} applied to the foliation 
$\sH$ on $X$ and $L:=A_\varepsilon+N$, we conclude
that $K_\sH + P + N + \varepsilon A  \equiv K_\sH-K_\sF+\varepsilon A$ is pseudo-effective for any positive rational number $\varepsilon >0$. It follows that $K_\sH-K_\sF$ is pseudo-effective.

\bigskip

\noindent\textbf{Step 3}. We will show that
$(K_\sH-K_\sF)\cdot A^{\dim X-1}\le 0$. 
Note that $\psi \colon Z \to Y$ is equidimensional.
Thus, by \ref{pullback_foliations},
there is an effective divisor $R$ on $X$ such that
$$K_\sH-K_\sF=-(\phi^*K_\sG
+R).$$
To prove that $(K_\sH-K_\sF)\cdot A^{\dim X-1}\le 0$, it is enough to show that
$\phi^*K_\sG\cdot A^{\dim X-1}\ge 0$. We argue by contradiction, and we assume that
$\phi^*K_\sG\cdot A^{\dim X-1}< 0$.
By Proposition \ref{proposition:HN_on_the_base},
there exists a saturated subsheaf $\sE \subseteq \sG$ satisfying the following property.
Let $C$ be a general complete intersection curve of type 
$m\nu^*A$ with $m$ sufficiently large. Then
${\psi^*\sE}_{|C}$ is the maximally destabilizing subsheaf of $\psi^*\sG_{|C}$. 
By our current assumption, we have $\deg({\psi^*\sG}_{|C})>0$ so that 
we must 
have $\deg({\psi^*\sE}_{|C})>0$. This implies that ${\psi^*\sE}_{|C}$ is ample.  
We conclude that $\sE_{|B}$ is an ample vector bundle, where $B:=\psi(C)$. 
By \cite[Proposition 30]{kebekus_solaconde_toma07}, we have
$\textup{Hom}_C({\psi^*\sE}_{|C}\otimes {\psi^*\sE}_{|C},{\psi^*\sG}_{|C}/{\psi^*\sE}_{|C})=0$
and thus 
$\textup{Hom}_Y(\sE\otimes \sE,\sG/\sE)=0$. Since $\sG$ is closed under the Lie bracket, it follows that $\sE$ is a foliation. 
By Theorem \ref{thm:BM}, we conclude that the leaf of $\sE$ through any point of $B$ is algebraic. But this contradicts $(\Diamond)$ above.
This proves that $\phi^*K_\sG\cdot A^{\dim X-1}\ge 0$, and 
$(K_\sH-K_\sF)\cdot A^{\dim X-1}\le 0$.

\bigskip

\noindent\textbf{Step 4}. By Steps 2 and 3, we must have 
$(K_\sH-K_\sF)\cdot A^{\dim X-1} = 0$. 
From \cite[Proposition 6.5]{greb_toma}, we conclude
that $$K_\sH\equiv K_\sF.$$
This proves (2). Note also that we must have $\phi^*K_\sG\equiv 0$, proving (1).
\bigskip

\noindent\textbf{Step 5}. 
Let $\sH_Z$ be the folitation on $Z$ induced by $\sH$. We show that 
$$K_{\sH_Z}\sim \nu^*K_\sH .$$
Note that $K_{\sH_Z}=K_{Z/Y}-R(\psi)$ by \ref{pullback_foliations}.
Up to replacing $P$ with $K_\sH-N$ if necessary, we may assume that 
$K_\sH=P+N$. By \ref{family_leaves},
there is an effective Weil $\bQ$-divisor $B$ on $Z$ such that 
$$K_{\sH_Z}+B \sim \nu^* K_\sH,$$
and hence 
$$K_{\sH_Z}+\nu^*P+\nu^*N \sim -B.$$
Note that $\textup{Supp}(B)\cap F =\emptyset$.
The same argument used in Step 2 shows that 
$K_{\sH_Z}+\nu^*P+\nu^*N$ is pseudo-effective.
One only needs to replace the use of Proposition \ref{prop:pseudo_effective}
with Corollary \ref{cor:pseudo_effective2}.
Since $K_{\sH_Z}+\nu^*P+\nu^*N\sim-B$, we conclude that
$B=0$. This completes the proof of the proposition.
\end{proof}

\begin{lemma}\label{lemma:compact_leaf}
Let $X$ be a complex projective manifold, and let $\sF\subset T_X$ be a foliation. 
Suppose that either $\sF$ is regular, or that $\sF$ has a compact leaf. Then 
the algebraic part $\sH$ of $\sF$ has a compact leaf.
\end{lemma}

\begin{proof}
By \cite[Section 2.4]{loray_pereira_touzet},
there exist a normal projective variety $Y$, a dominant rational map $\phi\colon X \dashrightarrow Y$, and a foliation $\sG$ on $Y$ such that the following holds: 
\begin{enumerate}
\item[(a)] there is no positive-dimensional algebraic subvariety passing through a general point of $Y$ that is tangent to 
$\sG$; and
\item[(b)] $\sF$ is the pull-back of $\sG$ via $\phi$. 
\end{enumerate}
Note that $\sH$ is the foliation on $X$ induced by $\phi$. After replacing $Y$ with a birationally equivalent variety, we may assume that $Y$ is the family of leaves of $\sH$. 
Let $Z$ be the normalization of the universal cycle over $Y$.
It comes with morphisms $\psi\colon Z\to Y$, and 
$\nu\colon Z \to X$ (see \ref{family_leaves}).

To prove the lemma, we have to show that $\psi\big(\textup{Exc}(\nu)\big)\subsetneq Y$. 
We argue by contradiction and assume that 
$\psi\big(\textup{Exc}(\nu)\big) = Y$.

Let $F$ be the closure of a general fiber of $\phi$. 
Denote by $S$ the (possibly empty) singular locus of $\sF$.
Then $F \cap S = \emptyset$.  
Indeed, suppose that $S \neq \emptyset$, and consider a compact $L$ leaf of $\sF$. Pick 
$y \in Y$ such that $Z_y\cap L\neq \emptyset$. Then $Z_y \subset L$, and hence $Z_y \cap S =\emptyset$. 
Thus, if $y'$ is sufficiently close to $y$, then $Z_{y'} \cap S = \emptyset$, proving our claim.

By Zariski's Main Theorem, there is an integral (complete) curve $C \subset Z$ with 
$\dim \nu(C)=0$ and $\nu(C) \in F$. Set $B:=\psi(C)$, and note that 
$\dim B=1$. Then $\nu(\psi^{-1}(B)) \supset F$ 
is tangent to $\sF$ since 
$\sF$ is regular in a neighborhood of $F$, and
$\dim \nu(\psi^{-1}(B)) = \dim F +1$. This implies that $B$ is tangent to $\sG$, yielding a contradiction and proving the lemma. 
\end{proof}

We believe that the following result will be useful when considering 
arbitrary (regular) foliations.

\begin{prop}
Let $X$ be a complex projective manifold, and let $\sF\subset T_X$ be a foliation. 
Let $\sH$ be the algebraic part of $\sF$. Suppose that
$\sH$ is induced by an almost proper map
$\phi \colon X \dashrightarrow Y$, and let $\sG$ be the foliation on $Y$ such that
$\sF=\phi^{-1}\sG$.
Then $\mu_A^{\textup{max}}\big((\phi^*\sG)^{**}\big)\le 0$ for any ample divisor $A$ on $X$.
\end{prop}

\begin{proof}
This follows easily from Steps 1 and 3 of the proof of Proposition \ref{proposition:regular_versus_algebraically_integrable_foliation} using 
the Mehta-Ramanathan theorem.
\end{proof}

\section{Proofs}

In this section we prove Theorem \ref{thm:main} and Corollary \ref{cor:kodaira_codim1}. 

\begin{thm}\label{thm:main2}
Let $X$ be a complex projective manifold, and let $\sF\subset T_X$ be a 
codimension $q$ foliation
with $0 < q < \dim X$. 
Suppose that the algebraic part $\sH$ of $\sF$ has a compact leaf.
Suppose furthermore that $-K_\sF \equiv A+E$ where 
$A$ is a $\mathbb{Q}$-ample $\mathbb{Q}$-divisor and
$E$ is an effective $\mathbb{Q}$-divisor.
Then $(X,E)$ is not log canonical.
\end{thm}

\begin{proof}Argue by contradiction, and assume that $(X,E)$ is log canonical.
There exist a dense open subset $X^\circ \subset X$, a normal variety $Y^\circ$, and a proper morphism $\phi^\circ\colon X^\circ \to Y^\circ$ such that 
$\sH$ is the foliation on $X$ induced by $\varphi^\circ$.
By Proposition \ref{proposition:regular_versus_algebraically_integrable_foliation},
we have $K_\sH\equiv K_\sF$.
Note that $\dim Y^\circ < \dim X^\circ$ since 
$K_\sF\not\equiv 0$. Set $A'=-K_\sH-E$. Then $A'$ is $\mathbb{Q}$-ample, and $-K_\sH=A'+E$. 

Let $y$ be a general point in $Y$, and denote by $X_y$ the corresponding fiber of $\phi$. The pair 
$(X_y,E_{|X_y})$ is log canonical by \cite[Theorem 4.8]{kollar97}. Also, note that $X_y$ is smooth. Let $0< \varepsilon\ll 1$ a rational number such that $A'+\varepsilon E$ is $\mathbb{Q}$-ample, and 
such that $(X_y,(1-\varepsilon) E_{|X_y})$ is klt.
Then $-K_\sH=(A'+\varepsilon E)+(1-\varepsilon) E$.
But 
this contradicts Proposition \ref{proposition:thm_alg_int} applied to $\sH$.
\end{proof}

Note that Theorem \ref{thm:main} and Corollary \ref{cor:kodaira_codim1} are consequences of
Theorem \ref{thm:mainbis} and Corollary \ref{cor:kodaira_codim1_bis} respectively using Lemma \ref{lemma:compact_leaf}.

\begin{thm}\label{thm:mainbis}
Let $X$ be a complex projective manifold, and let $\sF\subset T_X$ be a 
codimension $q$ foliation
with $0 < q < \dim X$. 
Suppose that the algebraic part of $\sF$ has a compact leaf.
Then $-K_\sF$ is not nef and big.
\end{thm}

\begin{proof}
Write $-K_\sF=A+E$ where 
$A$ is a $\mathbb{Q}$-ample $\mathbb{Q}$-divisor, and 
$E$ is an effective $\mathbb{Q}$-divisor. Let $0<\varepsilon<1$ be a rational number such that 
$(X,\varepsilon E)$ is klt. Note that $A_\varepsilon:=-K_\sF-\varepsilon E$ is $\mathbb{Q}$-ample.
This contradicts Theorem \ref{thm:main2}, proving Theorem \ref{thm:main}.
\end{proof}

\begin{cor}\label{cor:kodaira_codim1_bis}
Let $X$ be a complex projective manifold, and let $\sF\subset T_X$ be a 
codimension $1$ foliation. 
Suppose that the algebraic part of $\sF$ has a compact leaf.
Suppose 
furthermore
that $-K_\sF$ is nef. Then $\kappa(X,-K_\sF) \le \dim X-1$.
\end{cor}

\begin{proof}
By Theorem \ref{thm:mainbis}, we must have $\nu(-K_\sF) \le \dim X -1$. Thus
$\kappa(X,-K_\sF)\le \nu(-K_\sF)\le \dim X-1$ by \cite[Proposition 2.2]{kawamata}.
\end{proof}

\begin{question}\label{question:kodaira_versus_rank}
Given a regular foliation $\sF$ of rank r on a complex projective manifols with $-K_\sF$ nef, do we have 
$\kappa(X,-K_\sF) \le r$?
\end{question}

\section{Foliations with nef and abundant anti-canonical bundle}

In this section, we answer Question \ref{question:kodaira_versus_rank} under an additional assumption. First, we recall the definition and basic properties of nef and abundant divisors. 

\begin{say}[Abundant nef divisors]\label{abundant_nef_divisors}
Let $P\not\equiv 0$ be a nef $\mathbb{Q}$-Cartier $\mathbb{Q}$-divisor on a normal projective variety $X$.
Recall that the \emph{numerical dimension} of $P$ is the largest positive integer
$k$ such that $P^k\not\equiv 0$. 
We have $\kappa(X,P) \le \nu(P)$ by \cite[Proposition 2.2]{kawamata}, and we say that $P$ is \emph{abundant} if
equality holds. 

By \cite[Proposition 2.1]{kawamata}, $P$ is abundant if and only if there is a diagram of normal 
projective varieties

\centerline{
\xymatrix{
& Z\ar[ld]_{\nu}\ar[rd]^{f} & \\
X & & T \\
}
}
\noindent and a nef and big $\mathbb{Q}$-Cartier $\mathbb{Q}$-divsor $D$ on $T$ such that
$\nu^*P\sim_\mathbb{Q} f^*D$, where $\nu$ is a birational morphism, and $f$ is surjective. 
\end{say}

We will need the following easy observation.

\begin{lemma}\label{lemma:nef_abudant_divisor}
Let $P$ be a nef and abundant $\mathbb{Q}$-Cartier $\mathbb{Q}$-divisor on a normal projective variety $X$. 
With the above notation, write $D=A+E$ where 
$A$ is a $\mathbb{Q}$-ample $\mathbb{Q}$-divisor, and 
$E$ is an effective $\mathbb{Q}$-divisor. 
Let $F \subseteq X$ be a subvariety not contained in
$\nu\Big(\textup{Exc}(\nu)\cup f^{-1}\big(\textup{Supp}(E)\big)\Big)$. Then $P_{|F}$ is nef and abundant, and 
$\nu(P_{|F})=\dim f(\tilde{F})$, where $\tilde{F}$ denotes the proper transform of $F$ in $Z$.
\end{lemma}

\begin{proof}
Note that 
$f(\tilde{F}) \not \subset E$, and thus $D_{|f(\tilde{F})}$ is nef and big. The lemma follows easily.
\end{proof}

The following observation will prove to be crucial.

\begin{lemma}\label{lemma:nef_abundant_restriction_leaf_foliation}
Let $X$ and $Y$ be normal projective varieties, and let $\phi \colon X \dashrightarrow Y$ be an almost proper map with general fiber $F$. Let $\sH$ be the induced foliation on $X$. 
Suppose that $\sH$ is $\mathbb{Q}$-Gorenstein.
Suppose furthermore that $-K_\sH \equiv P$ where $P$ is nef and abundant, and that ${-K_\sH}_{|F}=-K_F$ is nef and abundant. Then $\nu(-K_\sH) \le \nu (-K_F)$. 
\end{lemma}

\begin{proof}
Let $\psi\colon Z\to Y$ be the family of leaves of $\sH$, and let $\nu\colon Z \to X$ be the natural morphism (see \ref{family_leaves}).
By Lemma \ref{lemma:reduced_fibers_codimension_1},
there exists a finite surjective morphism
$\mu_1\colon Y_1 \to Y$ with $Y_1$ normal and connected satisfying the following property. If $Z_1$ denotes the normalization of the product $Y_1 \times_Y Z$, then the induced morphism $\psi_1\colon Z_1 \to Y_1$ has reduced fibers over codimension one points in $Y_1$. 
Denote by $\nu_1\colon Z_1 \to Z$ the natural morphism.

By Proposition \ref{proposition:regular_versus_algebraically_integrable_foliation} and \ref{pullback_foliations}, we have
$$K_{Z/Y}-R(\psi)\sim\nu^*K_\sH.$$
A straightforward computation shows that 
$$\nu_1^*\big(K_{Z/Y}-R(\psi)\big)=K_{Z_1/Y_1}.$$
We conclude that $-K_{Z_1/Y_1}$ is $\mathbb{Q}$-Cartier, and that
$-K_{Z_1/Y_1}\equiv P_1$, where  $P_1:=(\nu_1\circ\nu)^* P$.

By \cite[Proposition 2.1]{kawamata} applied to $P_1$,
there exist a surjective morphism $f_3\colon Z_3 \to T_3$
of normal projective varieties,
a birational morphism
$\nu_3\colon Z_3 \to Z_1$, and a nef and big $\mathbb{Q}$-Cartier $\mathbb{Q}$-divsor 
$D_3$ on $T_3$ such that
$\nu_3^*P_1\sim_\mathbb{Q} f_3^*D_3$.
We obtain a commutative diagram

\centerline{
\xymatrix{
Z_3 \ar[rr]^{\nu_3,\text{ birational}}\ar[d]^{f_3} 
 & & Z_1 \ar[rr]^{\nu_1,\text{ finite}}\ar[d]^{\psi_1} & & Z \ar[d]^{\psi}\ar[rr]^{\nu,\text{ birational}} && X \\
T_3 & & Y_1 \ar[rr]_{\mu_1,\text{ finite}} && Y. && \\
}
}
\noindent 
We view $F$ as a fiber of $\psi$.
Let $F_1$ be a fiber of $\psi_1$ such that $\nu_1(F_1)=F$.
Note that
the restriction of $\nu_1$ to 
$F_1$ induces an isomorphism $F_1\cong F$.
Let $\tilde{F}_1$ be the proper transform of $F_1$ in $Z_3$. We also have that 
${P_1}_{|F_1}= P_{|F}\equiv -K_F$.
By Lemma \ref{lemma:nef_abudant_divisor} above,
we obtain that
$\dim f_3(\tilde{F}_1) = \nu({P_1}_{|F_1})=\nu(-K_F)$.

We argue by contradiction, and assume that $\dim T_3 = \nu(-K_\sH) > \nu (-K_F)$.
Let $Y_2 \subset Y_1$ be a general complete intersection curve, and
set $Z_2:=\psi_1^{-1}(Y_2) \subset Z_1$. 
It comes with a morphism $\psi_2\colon Z_2 \to Y_2$.
Let $\tilde{Z_2}$ be the proper transform of $Z_2$ in $Z_3$.
Then $\dim f_3(\tilde{Z_2}) = \dim f_3(\tilde{F}) + 1= \nu(-K_F) +1$, and thus
$\nu\big( {-K_{Z_1/Y_1}}_{|Z_2} \big)= \nu({P_1}_{|Z_2})=\nu(-K_F) +1$ by Lemma \ref{lemma:nef_abudant_divisor} again.
By the adjunction formula, we have ${K_{Z_1/Y_1}}_{|Z_2}\sim K_{Z_2/Y_2}$. This gives
$$\nu\big( -K_{Z_2/Y_2} \big)=\nu(-K_F) +1.$$
By proposition \ref{prop:singularities_foliations2}, 
$Z_2$
has canonical singularities, and 
by \cite[Theorem 1.1]{fujino_kawamata}, 
$-K_{Z_2/Y_2}$ is $\psi_2$-semi-ample. Therefore, there exist a $Y_2$-morphism
$f_2\colon Z_2 \to T_2$ with connected fibers onto a normal projective variety, and a $\mathbb{Q}$-Cartier $\mathbb{Q}$-divisor $M$ on $T_2$ ample over $Y_2$ such that
\begin{equation*}
-K_{Z_2/Y_2}\sim_\mathbb{Q} f_2^*M.
\end{equation*}
\noindent By Ambro's canonical bundle formula \cite[Theorem 4.1]{ambro_cbf}, there exists an effective $\mathbb{Q}$-divisor $B_2$ on $T_2$ such that $(T_2,B_2)$ is klt and
\begin{equation*}
K_{Z_2}\sim_\mathbb{Q} f_2^*(K_{T_2}+B_2).
\end{equation*}
Thus
\begin{equation*}
K_{Z_2/Y_2}\sim_\mathbb{Q} f_2^*(K_{T_2/Y_2}+B_2),
\end{equation*}
and hence
$-(K_{T_2/Y_2}+B_2)$ is nef with 
$$\nu\big(-(K_{T_2/Y_2}+B_2)\big)=\nu(-K_{Z_2/Y_2})  =\nu(-K_F) +1=\dim T_2.$$ 
But this contradicts \cite[Theorem 5.1]{fano_fols} (see also Proposition \ref{proposition:thm_alg_int}), 
completing the proof of the lemma.
\end{proof}

\begin{cor}\label{cor:codim1_abundant_max_kod_dim_1}
Let $X$ be a complex projective manifold, let $Y$ be a normal projective variety, and let $\phi \colon X \dashrightarrow Y$ be an almost proper map with general fiber $F$. Let $\sH$ be the induced foliation on $X$. 
Suppose that $\sH$ is $\mathbb{Q}$-Gorenstein.
Suppose furthermore that $-K_\sH$ is nef and abundant with $\nu(-K_\sF)=\dim X - dim Y$. 
Then 
$-K_F$ is nef and big, and $-K_\sF$ is semiample.
\end{cor}

\begin{proof}
Note first that $-K_F$ is nef and abundant by Lemma \ref{lemma:nef_abudant_divisor}. Then, apply
Lemma \ref{lemma:nef_abundant_restriction_leaf_foliation} to conclude that $-K_F$ is nef and big. In particular, $-K_F$ is semiample.

By \cite[Proposition 2.1]{kawamata} applied to $-K_\sF$,
there exist a surjective morphism $f \colon Z\to T$
of normal projective varieties,
a birational morphism
$\nu\colon Z \to X$, and a nef and big $\mathbb{Q}$-Cartier $\mathbb{Q}$-divsor 
$D$ on $T$ such that
\begin{equation}\label{eq:nef_abundant}
-\nu^*K_\sF\sim_\mathbb{Q} f^*D.
\end{equation}
Let $\tilde{F}$ denote the proper transform of $F$ in $Z$.
The restriction $f_{|\tilde{F}} \colon \tilde{F} \to T$ of $f$ to $\tilde{F}$ is generically finite
by Lemma \ref{lemma:nef_abudant_divisor}, and \eqref{eq:nef_abundant} reads 
$$-\nu_{|\tilde{F}}^*K_F\sim_\mathbb{Q}f_{|\tilde{F}}^*D.$$
Since $-K_F$ is semiample, we conclude that $D$ is semiample as well, and hence so is $-K_\sF$.
\end{proof}

\begin{cor}\label{cor:codim1_abundant_max_kod_dim}
Let $X$ be a complex projective manifold, and let
$\phi\colon X \to Y$ be a morphism onto a smooth complete curve with connected general fiber $F$. Denote by $\sF$ the induced foliation on $X$. Suppose that $-K_\sF$ is nef and $\kappa(X,-K_\sF)=\dim X -1$. Then 
$-K_F$ is nef and big, and $-K_\sF$ is semiample.
\end{cor}

\begin{proof}
Note that $-K_\sF$ is nef and abundant by Theorem \ref{thm:main}. 
The claim then follows from Corollary \ref{cor:codim1_abundant_max_kod_dim_1}.
\end{proof}

The proof of the next lemma is similar to that of Lemma \ref{lemma:nef_abundant_restriction_leaf_foliation}.

\begin{lemma}\label{lemma:nef_abundant_restriction_leaf_fibration}
Let $\phi\colon X \to Y$ be a morphism of normal projective varieties with general fiber $F$. Suppose that there exists a $\mathbb{Q}$-Cartier divisor $K$ on $X$ such that 
$K_{|\phi^{-1}(Y_{\textup{ns}})}\sim K_{\phi^{-1}(Y_{\textup{ns}})/Y_{\textup{ns}}}$. Suppose furthermore that
$K$ is nef and abundant. Then $\nu(-K) \le \nu (-K_F)$. 
\end{lemma}

The same argument used in the proof of Corollary \ref{cor:codim1_abundant_max_kod_dim}
shows that the following holds. One only needs to replace the use of 
Theorem \ref{thm:main} with \cite[Theorem 5.1]{fano_fols}, and the use of  
Lemma \ref{lemma:nef_abundant_restriction_leaf_foliation} 
with Lemma \ref{lemma:nef_abundant_restriction_leaf_fibration}.

\begin{cor}\label{lemma:codim1_abundant_morphism_max_kod_dim}
Let $X$ be a complex projective manifold, and let
$\phi\colon X \to Y$ be a morphism onto a smooth complete curve with connected general fiber $F$. 
Suppose that $-K_{X/Y}$ is nef and $\kappa(X,-K_{X/Y})=\dim X -1$. Then 
$-K_F$ is nef and big, and $-K_{X/Y}$ is semiample.
\end{cor}

Note that Theorem \ref{thm:main3} is an immediate consequence of Theorem \ref{thm:main3bis}
and Lemma \ref{lemma:compact_leaf}.

\begin{thm}\label{thm:main3bis}
Let $X$ be a complex projective manifold, 
and let $\sF\subset T_X$ be a 
codimension $q$ foliation
with $0 < q < \dim X$. 
Suppose that the algebraic part $\sH$ of $\sF$ has a compact leaf.
Suppose 
furthermore
that $-K_\sF $ is nef and abundant.
Then $\kappa(X,-K_\sF) \le \dim X-q$, and equality holds only if 
$\sF$ is algebraically integrable.
\end{thm}

\begin{proof}By assumption, $\sH$ is induced by an almost proper map $\phi\colon X \dashrightarrow Y$. By
Proposition \ref{proposition:regular_versus_algebraically_integrable_foliation}, we have
$K_\sH\equiv K_\sF$. Let $\sG$ be the foliation on $Y$ such that 
$\sF=\phi^{-1}\sG$. There exists an open set $U \supset F$ contained in $X$ such that 
$\sH_{|U}$ is a subbundle of $\sF_{|U}$, $\sG_{|U}$ is locally free,
and 
$({\sF/\sH})_{|U}\cong {\phi^*\sG}_{|U}$.
In particular, we have ${K_\sF}_{|F}\sim {K_\sH}_{|F}\sim K_F$, and hence 
$-K_F$ is nef and abundant by
Lemma \ref{lemma:nef_abudant_divisor}.
Then
$$
\begin{array}{ccll}
\kappa(X,-K_\sF) & = & \nu(-K_\sF) & \text{ since }-K_\sF \text{ is nef and abundant }\\
& = &  \nu(-K_\sH) & \text{ since } K_\sH\equiv K_\sF\\
& \le & \nu(-K_F) & \text{ by Lemma \ref{lemma:nef_abundant_restriction_leaf_foliation}} \\
& \le & \dim F = \textup{rank }\sH & \\
& \le & \textup{rank }\sF =\dim X -q  & \text{ since } \sH \subset \sF.
\end{array}
$$

Suppose that $\kappa(X,-K_\sF)=\dim X -q$. Then we must have 
$\textup{rank }\sH = \textup{rank }\sF$. Thus $\sH=\sF$, and hence $\sF$ is algebraically integrable.
This completes the proof of Theorem
\ref{thm:main3bis}.
\end{proof}

\begin{cor}\label{cor:integrability}
Let $X$ be a complex projective manifold, and let $\sF\subset T_X$ be a 
codimension $1$ foliation. Suppose that either $\sF$ is regular, or that $\sF$ has a compact leaf.
Suppose furthermore
that $-K_\sF$ is nef and $\kappa(X,-K_\sF) = \dim X-1$. Then 
$\sF$ is algebraically integrable.
\end{cor}

\begin{proof}
Note that $-K_\sF$ is abundant by Corollary \ref{cor:kodaira_codim1}. 
The claim then follows from Theorem \ref{thm:main3}.
\end{proof}

The remainder of the present section is devoted to the proof of 
Theorem \ref{thm:main4}.

\begin{lemma}\label{lemma:smooth_morphism}
Let $X$ be a complex projective manifold, and let $\sF\subset T_X$ be a 
regular codimension $1$ foliation. 
Suppose 
that $-K_\sF$ is nef and $\kappa(X,-K_\sF) = \dim X-1$.
Then $\sF$ is induced by a smooth morphism $X \to Y$ onto a smooth complete curve $Y$ of genus
$g(Y)=h^1(X,\sO_X)$.
\end{lemma}

\begin{proof}
By Corollary \ref{cor:integrability}, $\sF$ is algebraically integrable. So let $\phi\colon X \to Y$ be a first integral, with general fiber $F$. By Corollary \ref{cor:codim1_abundant_max_kod_dim}, $-K_F$ is nef and big. Let $F_1=m F_0$ be any fiber of $\phi$, where $m$ is a positive integer. By the adjunction formula, we have 
$K_{F_0}=(K_X+F_0)_{|F_0}$, and hence
\begin{multline*}
K_{F_0}^{\dim X-1}=(K_X+F_0)^{\dim X-1} \cdot F_0
=\big(K_{X/Y}-(m-1)F_0\big)^{\dim X-1} \cdot F_0 \\
=K_\sF^{\dim X-1} \cdot F_0
=\frac{1}{m}K_\sF^{\dim X-1} \cdot F_1
=\frac{1}{m}K_\sF^{\dim X-1} \cdot F
=\frac{1}{m}K_{F}^{\dim X-1}>0.
\end{multline*}
Moreover, $K_{F_0}=(K_X+F_0)_{|F_0}\sim \big(K_X-(m-1)F_0\big)_{|F_0}\sim {K_\sF}_{|F_0}$
is nef by assumption.
In particular, $\pi_1(F_0)=\{1\}$ (see \cite{zhang_rcc}). 
By the holomorphic version of Reeb stability theorem, we conclude that $\phi$ is a smooth morphism.

By the Kawamata-Viehweg vanishing theorem, we have $R^1\phi_*\sO_X=0$. The Leray spectral sequence then yields 
$h^1(X,\sO_X)=h^1(Y,\sO_Y)$. This completes the proof of the lemma.
\end{proof}

\begin{proof}[Proof of Theorem \ref{thm:main4}]
Let $X$ be a complex projective manifold with $h^1(X,\sO_X)=0$, and let $\sF\subset T_X$ be a 
codimension $1$ regular foliation. 
Suppose 
that $-K_\sF$ is nef and $\kappa(X,-K_\sF) = \dim X-1$.

By Lemma \ref{lemma:smooth_morphism}, $\sF$ is induced by a smooth morphism 
$\phi\colon X\to \mathbb{P}^1$.
In particular, $K_\sF=K_{X/\mathbb{P}^1}$.
Now, observe that $-K_X = -K_{X/\mathbb{P}^1}-\phi^*K_{\mathbb{P}^1}$
is nef, and that
$$(-K_X)^{\dim X}=(-K_{X/\mathbb{P}^1}-\phi^*K_{\mathbb{P}^1})^{\dim X}=2\dim X (-K_F)^{\dim F} > 0.$$
In other words, $X$ is a weak Fano manifold. 
Let $F$ be a general fiber of $\phi$. Let
$R=\mathbb{R}_{\ge 0}[\ell] \subset \textup{NE}(X)$
be an extremal ray with $F \cdot \ell>0$, and let $\psi_R\colon X \to Z$ be the corresponding contraction.
Note that any fiber of $\psi_R$ has dimension $\le 1$. Thus, $Z$ is smooth and either 
$\psi_R\colon X \to Z$ is the blow up of a codimension $2$ subvariety, or $\psi_R\colon X\to Z $ is a conic bundle by 
\cite[Theorem 1.2]{wisn_crelle}.

Suppose first that $\psi_R\colon X \to Z$ is the blow up of a codimension $2$ subvariety. Then, up to replacing $\ell$ with a numerically equivalent curve on $X$, we may assume that $-K_X\cdot \ell =1$. On the other hand,
$$-K_X \cdot \ell = (-K_{X/\mathbb{P}^1}-\phi^*K_{\mathbb{P}^1})\cdot \ell \ge 2 F \cdot \ell \ge 2,$$ yielding a contradiction. This proves that $\psi_R\colon X \to Z$ is a conic bundle. Arguing as above, we conclude that
$\psi_R\colon X \to Z$ is a smooth conic bundle, and that $F \cdot \ell =1$.
This implies that the morphism $\phi \times \psi_R \colon X \to \mathbb{P}^1 \times Z$ is birational. On the 
other hand, it is finite, and hence $X \cong \mathbb{P}^1 \times F$, proving the theorem. 
\end{proof}

\begin{question}
Given a smooth morphism $X \to Y$ onto a smooth complete curve $Y$ with $g(Y)\ge 1$ such that 
$-K_{X/Y}$ is nef and $\kappa(X,-K_{X/Y})=\dim X-1$, do we have $X \cong Y\times F$?
\end{question}


\providecommand{\bysame}{\leavevmode\hbox to3em{\hrulefill}\thinspace}
\providecommand{\MR}{\relax\ifhmode\unskip\space\fi MR }
\providecommand{\MRhref}[2]{%
  \href{http://www.ams.org/mathscinet-getitem?mr=#1}{#2}
}
\providecommand{\href}[2]{#2}

\end{document}